\documentclass[12pt]{article}
\usepackage[a4paper,hmargin=2cm,vmargin=3cm]{geometry}
\usepackage[T1]{fontenc}
\usepackage[utf8]{inputenc}
\usepackage{color}
\usepackage{amsmath}
\usepackage{amsthm}
\usepackage{amssymb}
\usepackage{graphicx}
\usepackage[style=alphabetic,maxnames=10,giveninits=true,backend=biber,doi=false,url=false]{biblatex}
\usepackage[shortlabels]{enumitem}
\usepackage{hyperref}

\definecolor{notefontcolor}{rgb}{0.800781, 0.800781, 0.800781}
\definecolor{grey30}{rgb}{0.7,0.7,0.7}

\hypersetup{unicode=true,
  colorlinks=true, citecolor=blue, linkcolor=blue,
  pdftitle=On the trivial regime of mixed p-spin spherical Hamiltonians with external field,
  pdfauthor={D. Belius, J. Černý, S. Nakajima, M. Schmidt}}

\DeclareSourcemap{
  \maps[datatype=bibtex]{
    \map{
      \step[fieldset=issn, null]
    }
  }
}
\numberwithin{equation}{section}

\theoremstyle{plain}
\newtheorem{theorem}{Theorem}[section]
\newtheorem{lemma}[theorem]{Lemma}
\newtheorem{proposition}[theorem]{Proposition}

\theoremstyle{remark}
\newtheorem{remark}[theorem]{Remark}

\newcommand{\1}{\boldsymbol 1}

\newcommand{\N}{{\mathbb{N}}}
\newcommand{\Z}{{\mathbb{Z}}}

\DeclareMathOperator{\Vol}{Vol}
\DeclareMathOperator{\sgn}{sgn}

\newcommand{\radconv}{r}

\newcommand{\e}{\varepsilon}

\newcommand{\R}{{\mathbb{R}}}
\renewcommand{\P}{{\mathbb P}}
\newcommand{\E}{\mathbb E}

\newcommand{\dd}{\,\mathrm d}     % a straight d for differentials

\newcommand{\kE}{{\cal E}}
\newcommand{\lf}{\lfloor}
\newcommand{\rf}{\rfloor}

\def\ben#1{\begin{equation}#1\end{equation}}

\def\al#1{\begin{align*}#1\end{align*}}
\def\aln#1{\begin{align}#1\end{align}}
\newcommand{\GOE}{{\mathrm{GOE}}}
\newcommand{\Id}{{\mathbb{I}}}
\newcommand{\sph}{{\mathrm{sp}}}
\newcommand{\dr}{\partial_r}
\newcommand{\dN}{\dr}
\newcommand{\Max}{{\mathrm{max}}}

\definecolor{mypink1}{rgb}{0.858, 0.188, 0.478}
\newcommand{\xione}{\boldsymbol{\xi}}
\newcommand{\xipone}{\boldsymbol{\xi}'}
\newcommand{\xippone}{\boldsymbol{\xi}''}
\newcommand{\uN}{\mathbf{u}_N}

\begin{document}

\title{Triviality of the geometry of mixed $p$-spin spherical Hamiltonians with
    external field}
\author{David Belius, Jiří Černý, Shuta Nakajima, Marius Schmidt}

\date{\normalsize Department of Mathematics and Computer Science, University of Basel}

\maketitle

\begin{abstract}
  We study isotropic Gaussian random fields on the high-dimensional
  sphere with an added deterministic linear term, also known as mixed $p$-spin
  Hamiltonians with external field. We prove that if the external field
  is sufficiently strong, then the resulting function has trivial
  geometry, that is only two critical points. This contrasts with the
  situation of no or weak external field where these functions typically
  have an exponential number of critical points. We give an explicit
  threshold $h_c$ for the magnitude of the external field necessary for
  trivialization and conjecture $h_c$ to be sharp. The Kac-Rice formula
  is our main tool. Our work extends \cite{fyodorov2013high}, which
  identified the trivial regime for the special case of pure $p$-spin
  Hamiltonians with \emph{random} external field.
\end{abstract}

\section{Introduction}

Isotropic Gaussian random fields on the sphere are paradigmatic high
dimensional complex functions. Due to their appearance in \emph{spin
  glass} models in statistical physics, they are also known as mixed $p$-spin
spherical Hamiltonians. One manifestation of the complexity is the
presence, in general, of an exponentially large number of critical points
(this has been proven for the special case of \emph{pure $p$-spin
    Hamiltonians}
  \cite{fyodorov2013high,auffinger2013random,subag2017complexity} and
  their perturbations \cite{AB13,arous2020geometry} and is expected to be
  generic beyond these special cases). In this paper, we prove that in
the presence of a deterministic linear term (external field in the
  physics terminology) with strength above a certain threshold, the
geometry of such functions trivializes in the sense that the only
critical points of these random function are one maximum and one
minimum. This extends \cite{fyodorov2014topology} which exhibited the trivialization phenomenon for pure $2$-spin Hamiltonians, and \cite{fyodorov2013high} which identified
the trivial regime for pure $p$-spin Hamiltonians
with \emph{random} external field, and makes mathematically rigorous part
of the results of \cite{ros2019complex} which demonstrated triviality for pure
$p$-spin Hamiltonians with deterministic external field using physics
methods. Our result proves trivialization for any mixed $p$-spin spherical Hamiltonian, which includes pure $p$-spin Hamiltonians as a special case, as well as Hamiltonians with a Gaussian random external field (see the discussion below Theorem \ref{thm:sec_thm}).
We further characterize the energies and other properties of the unique maximizer
and minimizer.

We now introduce our model. Let $\xi$ be a series
\begin{equation}\label{eq:xi_def}
  \xi\left(x\right)=\sum_{p\ge1}a_{p}x^{p},\qquad a_p\ge0,
\end{equation}
with radius of convergence $\radconv >1$, such that $a_p >0$ for at least
one $p\ge 2$. Let $H_{N}$ be a centered Gaussian process (the
  Hamiltonian) on the open ball in $\mathbb R^N$ with radius
$\sqrt{\radconv}$ whose covariance is given by
\begin{equation}\label{eq:covariance_H_N}
  \E \big[H_{N}(\sigma)H_{N}(\sigma')\big]
  =N\xi(\sigma\cdot\sigma'),
  \qquad |\sigma|, |\sigma'| < \sqrt{\radconv}.
\end{equation}
We
are mostly interested in the behavior of the Hamiltonian $H_N$ restricted
to the unit sphere
$S_{N-1}=\left\{ \sigma \in \mathbb R^N:\left|\sigma\right|=1\right\}$.

Note that any covariance function of an isotropic Gaussian random field
on the sphere must depend only on the scalar product $\sigma\cdot\sigma'$,
and thus take the form $\xi(\sigma \cdot \sigma')$ for some function $\xi$.
By Schoenberg's theorem \cite{schoenberg1942positive}, the only such $\xi$
that give well-defined covariances on $S_{N-1}$ for all $N$ are those of
the form \eqref{eq:xi_def}. They thus represent a very general class of
covariances of isotropic random Gaussian fields on the sphere. If
$\xi\left(x\right)=a_{p}x^{p}$ for some $p\ge 2$, then we call $H_N$ a
pure $p$-spin Hamiltonian.

For $h\ge0$ and a deterministic sequence $\uN\in S_{N-1}$,  we consider
the Hamiltonian with external field $h\uN$
\begin{equation}\label{eq:Hh}
  H_{N}^{h}(\sigma)=H_{N}(\sigma)+Nh\uN\cdot\sigma.
\end{equation}
A critical point of $H_N^h$ on $S_{N-1}$ is a $\sigma \in S_{N-1}$ such that
\begin{equation*}
  \nabla_{\sph} H_N^h(\sigma)=0,
\end{equation*}
where $\nabla_{\sph}$ denotes the gradient in the spherical metric (that
  is the standard gradient projected on the tangent space of $S_{N-1}$ at
  $\sigma$). We further use $\dr H^h_N(\sigma )$ to denote the radial
derivative of $H_N^h$ at $\sigma $, $\nabla_{\sph}^{2}H_N^h(\sigma)$ the
spherical Hessian, and $\lambda_{\max}(\nabla_{\sph}^{2}H_{N}^h(\sigma ))$
its largest eigenvalue. Using the shorthand notation
  $\xione=\xi(1)$, $\xipone=\xi'(1)$, $\xippone=\xi''(1)$, our main
  result shows that the function $H_{N}^{h}\left(\sigma\right)$
  trivializes for $h^2 > \xippone - \xipone$
  and gives formulas describing the properties of this function
    at its unique maximizer.
\begin{theorem}
  \label{thm:main_thm}
  If $h^2 > \xippone-\xipone$, then
  \begin{equation}\label{eq:triviality}
    \lim_{N\to\infty}\mathbb{P}\left(\begin{array}{c}
        \text{The only critical points of }H_{N}^{h}\text{ are }\\
        \text{one maximum and one minimum}
    \end{array}\right)=1
  \end{equation}
  and, letting $\sigma^{*}$ be the global maximum of $H_{N}^{h}$
  \begin{align}
    \lim_{N\to\infty}
    \frac{1}{N}H_{N}^h(\sigma^*)
    &=\sqrt{\xipone+h^2},
    \label{eq:main_res_GS_energy}
    \\
    \lim_{N\to\infty}\sigma^{*}\cdot{\mathbf u}_{N}
    &= \frac{h}{\sqrt{\xipone+h^2}},
    \label{eq:main_res_overlap_with_ext_field}
    \\
    \lim_{N\to\infty}\frac{1}{N} \dr H_N^h(\sigma^*)
    &= \frac{\xipone+\xippone+h^2}{\sqrt{\xipone+h^2}},
    \label{eq:rad_deriv}
    \\
    \lim_{N\to\infty}\lambda_{\Max} (\nabla^2_{\sph} H_N^h(\sigma^*))
    & = 2\sqrt{\xippone} - \frac{\xipone+\xippone+h^2}{\sqrt{\xipone+h^2}},
    \label{eq:eval_hessian}
  \end{align}
  where the limits are in probability.
\end{theorem}
If $\xippone < \xipone$, then the conclusions hold for any
  $h\ge0$. On the other hand, if $\xippone \ge \xipone$, then the condition
  $h^2 > \xippone - \xipone$ is equivalent to $h > h_c$, where we define
  the threshold $h_c$ by
\begin{equation}\label{eq:hc}
  h_{c}=\sqrt{\xippone-\xipone}.
\end{equation}
Note that $\xippone \ge \xipone$ holds in particular if $a_1=0$, that is, if there is no random external field (see the discussion below
  Theorem~\ref{thm:sec_thm}).

The main step in proving Theorem~\ref{thm:main_thm} is a precise control
of the asymptotic behaviour of the expected number of critical points of
$H_N^h$ using the Kac-Rice formula, stated here as our second main result.
\begin{theorem}
  \label{thm:sec_thm}
  Let $\mathcal N_N$ be the number of critical points of $H^h_N$,
  \begin{equation}
    \label{eqn:NN}
    \mathcal N_N =
    \big|\{\sigma \in S_{N-1}: \nabla_{\sph} H_N^h (\sigma ) = 0\}\big|.
  \end{equation}
  \begin{enumerate}[(i)]
    \item If $h^2 >\xippone-\xipone$, then
    \begin{equation}\label{eq:expNtriv}
      \lim_{N\to\infty}\E \left[\mathcal{N}_N\right] = 2.
    \end{equation}
    \item If $h^2 \le \xippone - \xipone $
      (in the case $\xippone \ge \xipone$), then
    \begin{equation}
      \label{eq:main_res_triv2}
      \lim_{N\to\infty}
      \frac{1}{N}\ln\E[\mathcal{N}_{N}]=
      \begin{cases}
        \frac{1}{2}(\frac{h^2 }{h_c^2}-1-\ln{\frac{h^2}{h_c^2}}),
        \qquad& \text{if } h/h_c \in \left(\sqrt{\frac{ \xipone}{\xippone}},1\right],
        \\ \frac{1}{2}\ln{ \frac{ \xippone}{\xipone}}-\frac{h^2}{2\xipone},
        & \text{if } h/h_c \in \left[0,\sqrt{\frac{
              \xipone}{\xippone}}\right].
      \end{cases}
    \end{equation}
  \end{enumerate}
\end{theorem}

Observe that the triviality \eqref{eq:triviality} directly follows from
\eqref{eq:expNtriv} and Markov's inequality, since any differentiable
function on the sphere has at least two critical points, one global
maximum and one global minimum. The result \eqref{eq:main_res_triv2} also
gives the exponential rate of the expectation for $h < h_{c}$.

Note also that if $\xippone \ge \xipone$ (cf.~\eqref{eq:hc}), then the
right-hand side of \eqref{eq:eval_hessian} equals
$2\sqrt{\xippone} -  \frac{2\xippone+h^2-h_c^2}{\sqrt{\xippone+h^2-h_c^2}}$
and thus tends to zero as $h\downarrow h_{c}$, showing that the unique local
maximum becomes increasingly flat as the external field approaches the
critical value $h_{c}$ from above. Furthermore since $H_{N}$ and $-H_{N}$
are identical in law, statements similar to
\eqref{eq:main_res_GS_energy}--\eqref{eq:eval_hessian} for the unique
minimum follow, with the obvious change of sign.

When $a_1=0$, our claim
\eqref{eq:main_res_GS_energy} on the energy of the unique global maximum
coincides with (13) of Proposition 1 in \cite{chen2017parisi}. Our paper
  thus provides an alternative proof of this result. The method of
\cite{chen2017parisi} is very different, in that it uses the Parisi
formula to derive a general formula for
$\lim_{N\to\infty} \frac{1}{N}H_{N}^{h}\left(\sigma^{*}\right)$ (known as
  the \emph{ground state energy}), which is shown to simplify to the
right-hand side of \eqref{eq:main_res_GS_energy} when $h > h_c$. Using
this and a further approach the mathematically non-rigorous work
\cite{ros2019complex} argues for triviality precisely when $h>h_c$ in the
special case of pure $p$-spin Hamiltonians with deterministic external
field.

Fyodorov \cite{fyodorov2013high} proves \eqref{eq:expNtriv} (and thus
  \eqref{eq:triviality}) for pure $p$-spin Hamiltonians with \emph{random
  Gaussian} external field, that is for Hamiltonians of the form
$\tilde H_N(\sigma ) = H_N(\sigma) + h ( U_N \cdot\sigma)$, where $H_N$
is as above, and  where $ U_N$ is a centered Gaussian random vector in
$\mathbb R^N$ whose covariance is $N$ times the identity matrix, and
which is independent of $H_N$. The covariance of $\tilde H_N$ is then
$\mathbb E[ \tilde H_N(\sigma )\tilde H_N(\sigma' )]=N\tilde \xi (\sigma \cdot \sigma ')$
for $\tilde{\xi}(x)=h^2 x + \xi(x)$. Thus, since we allow $a_1>0$ in
\eqref{eq:covariance_H_N}, our results also cover the case of random
external field, or a combination of random and deterministic external
fields.

From the first mathematically rigorous uses of the Kac-Rice formula for spin glass
Hamiltonians in  \cite{fyodorov2004complexity, fyodorov2012, fyodorov2013high, auffinger2013random}
it has become a widely used tool in this context.
%in rigorous works on
%spin glasse
The work \cite{subag2017complexity} used it to compute the second moment
of $\mathcal{N}_N$ to obtain concentration of $\mathcal{N}_N$ for $h=0$
and $H_N$ a pure $p$-spin Hamiltonian (and in \cite{arous2020geometry}
  for perturbations thereof). The work \cite{fan2018tap} used it to count
so called TAP solutions, and \cite{subag2017geometry,arous2020geometry}
to compute free energies and study the Gibbs measure of certain
Hamiltonians. Furthermore \cite{arous2019landscape} used the Kac-Rice
formula for the similar problem of studying the \emph{complexity} (number
  of critical points at exponential scale)  of pure $p$-spin Hamiltonians
with a deterministic term of polynomial degree $p$.

In our proof, we follow Fyodorov \cite{fyodorov2013high} in using the
Kac-Rice formula to compute $\E [\mathcal{N}_N]$ and exploiting that the expected determinant of a shifted GOE matrix can be computed very precisely (see Lemmas \ref{lem:detGOE}, \ref{lem:rhoapprox}). Our
proof diverges from \cite{fyodorov2013high} in that all our computations
are for general $\xi$ rather than the pure $p$-spin covariance function $\xi(x)=x^p$, and, more importantly,
because when considering a deterministic external field one obtains from
the Kac-Rice formula an integral over two rather than one variables. To
find the asymptotic of the integral one must thus find explicit formulas
for the maximizers of a function of $\mathbb{R}^2$ rather than as in
\cite{fyodorov2013high} for a function of $\mathbb{R}$ (see
  Section~\ref{sec:opti}). The extra variable corresponds to the inner
product with the deterministic external field, whereas with random
external field the only variable of integration corresponds to the radial
derivative.

Though we do not prove it, there is a good reason to believe that the
threshold $h_c$ is sharp for the triviality \eqref{eq:triviality},
\eqref{eq:expNtriv}: Indeed \cite{chen2017parisi} shows that this is
precisely the threshold for the minimizer of their Parisi formula for the
ground state to be ``replica symmetric'', and replica calculations of
\cite{ros2019complex} demonstrate using physics methods that for
$h < h_{c}$ the quenched complexity
$\lim_{N\to\infty}\frac{1}{N}\ln\mathcal{N}_N$ is positive in the special
case of a pure $p$-spin Hamiltonian (but smaller than the right-hand side
  of \eqref{eq:main_res_triv2}, i.e. the ``quenched'' and ``annealed''
  averages do not coincide in the physics terminology).

Our work is a step on the way towards rigorously determining the
complexity of critical points for mixed $p$-spin Hamiltonians in general.
It would furthermore be interesting to investigate the ``physical''
consequences for the Gibbs measure of the triviality of the Hamiltonian.

%We will furthermore investigate in future work the consequences of triviality for the ``physical'' properties of spin glass models with external field.

%\DB{Add "Outlook"?}
%\begin{itemize}
%    \item Relevance for Gibbs measure, TAP triviality
%    \item Our quenched!=annealed
%    \item Relevance for stats application, machine learning
%\end{itemize}

%It is interesting to note that when $h\ge h_{c}$ the random variable
%$\mathcal{N}_N$ coincides with its expectation $\mathbb{E}\left[\mathcal{N}_N\right]$
%with high probability, that is the quenched and the annealed averages
%agree in the statistical physics terminology. By contrast for $0<h<h_{c}$,
%$\lim_{N\to\infty} \frac{1}{N}\ln\mathcal{N}_N$ is not expected to
%coincide with the right-hand side of \eqref{eq:main_res_triv2},
%as argued heuristically by \cite{ros2019complex}, so in this regime the quenched
%and annealed averages should be  different.

\paragraph{Structure of paper}

In Section~\ref{sec:prelim}, we introduce notation and recall some
results on random matrices. In~Section \ref{sec:exact}, we derive an
exact and essentially explicit formula for the mean number of critical
points of $H^h_N$ on the sphere. To this end, we employ the Kac-Rice
formula which in our setting reads (see
  e.g.~\cite[(12.1.4)]{adler2009random})
\begin{equation}
  \label{eq:KacRice}
  \mathbb{E}\left[\mathcal{N}_{N}\right]
  =  \int_{S_{N-1}} \E\left[\left|
    \det \nabla^2 H^{h}_N(\sigma)
    \right|\,\Big|\,
    \nabla H^{h}_N(\sigma)=0 \right]
  f_{\nabla_\sph
    H^h_N(\sigma)}(0)\,
  \dd \sigma,
\end{equation}
where $\dd\sigma$ is the area element on $S_{N-1}$ and
where $f_{\nabla_\sph H^h_N(\sigma)}$ is the density of
$\nabla_\sph H^h_N(\sigma)$. We also use a slightly more general version
restricting energy, radial derivative $x$, and overlap
  $\gamma $, with the external field
to an arbitrary measurable set. The upshot is an estimate of the form
\begin{equation*}
  \mathbb{E}\left[\mathcal{N}_{N}\right]
  = e^{o(N)} \int_{[-1,1] \times\R}
  \exp{\left(N F(x,\gamma)\right)}\dd \gamma  \dd x ,
\end{equation*}
where $F$ is defined in \eqref{eq:defF} and a precise asymptotic for the
term $e^{o(N)}$ is also provided. From this it is clear that the
asymptotic behaviour of $\E [\mathcal{N}_N]$ is closely connected to the
maximizers of $F$. Section \ref{sec:opti} is devoted to the explicit
computation of these maximizers via the solution of the critical point
equations for $F$. We will see that their behavior is different for
$h< h_c$ and $h> h_c$, see Proposition \ref{pro:maxF}. Knowledge of the
maximizers will allow us to verify \eqref{eq:main_res_triv2}, as well as
a weaker version of \eqref{eq:expNtriv}, namely that for $h>h_c$
\begin{equation*}
  \lim_{N\to\infty}N^{-1} \ln \mathbb{E}\left[\mathcal{N}_{N}\right]
  =0.
\end{equation*}
A detailed analysis of the subexponential contributions is conducted in
Section \ref{complexity_triviality} culminating in the proof of
\eqref{eq:expNtriv}. In Section \ref{sec:maxchar}, the claims
\eqref{eq:main_res_GS_energy}--\eqref{eq:eval_hessian} are proved using that any but the given energy,
radial derivative and overlap with external field have exponentially decaying
mean number of critical points, which implies the claims by Markov's
inequality.

\section{Preliminaries}
\label{sec:prelim}

In this section we introduce the notation that is used throughout the
paper and state few important results used in the proof of
Theorems~\ref{thm:main_thm}-\ref{thm:sec_thm}.

When considering the Hamiltonian and its
derivatives at a given $\sigma \in S_{N-1}$, we always express them in
the orthonormal basis $(\mathbf{e}_{i}(\sigma ))_{i=1}^N$ of $\mathbb R^N$
which is fixed so that $\mathbf{e}_N(\sigma)=\sigma$ and the vector
$\mathbf u_N$ lies in the plane spanned by $\mathbf e_1(\sigma )$ and
$\mathbf e_N(\sigma )$. Then $(\mathbf{e}_i(\sigma))_{i=1}^{N-1}$
is a basis for the tangent space of $S_{N-1}$ at $\sigma$. For a
sufficiently smooth function $f: \mathbb R^N\to \mathbb R$, we use
$\partial_i f(\sigma )$ to denote the standard derivative of $f$ in the
direction $\mathbf e_i(\sigma )$ at the point $\sigma$,
$\nabla f = (\partial_i f)_{i=1}^N$ stands for its Euclidean gradient,
and $\nabla^2 f(\sigma) = (\partial^2_{ij} f(\sigma ))_{i,j =1}^N$ for
its Euclidean Hessian. In this basis, $\partial_N f(\sigma)$ coincides
with
the radial derivative $\partial_r f(\sigma)$, the spherical gradient
$\nabla_\sph f (\sigma )$ is the restriction of the usual gradient to the
first $N-1$ coordinates,
\begin{equation}
  \label{eq:covgrad}
  \nabla_\sph f(\sigma ) = (\partial_i f (\sigma ))_{i=1}^{N-1},
\end{equation}
and the spherical Hessian satisfies
\begin{equation}
  \label{eq:covHess}
  \nabla^2_\sph f = \nabla^2 f \big|_\sph - \dN f \,\Id_{N-1}
  = (\partial^2_{ij} f - \delta_{ij} \dN f)_{i,j=1}^{N-1},
\end{equation}
where $\Id_N$ stands for the $N\times N$ identity matrix,
$\nabla^2 f |_\sph$ is the top left $(N-1)\times (N-1)$ submatrix
 of $\nabla^2 f$
and $\delta_{ij}$ is the Kronecker symbol.

If $\xi(x)$ has radius of convergence $r$ greater than one, then
$H_N(\sigma)$ is almost surely a smooth function on the open ball
$\{ \sigma: |\sigma| < \sqrt{\radconv} \} \subset \mathbb{R}^N$, so we
may speak of its Euclidean and spherical derivatives.

We write $a_n\sim b_n$ if $a_n/b_n\to 1$ as $n\to \infty$. For a random
variable $X$, we use $f_X$ to denote its density, if it exists.

For the evaluation of the determinant appearing in the Kac-Rice formula,
we will need few facts about GOE random matrices. Given $a>0$ and
$N \in \N$, we use  $\GOE_{N}(a)$ to denote a $N\times N$ symmetric
random matrix whose entries $A_{ij}$, $1\le i\le j\le N$, are independent
normal random variables with mean $0$ and variance
\begin{equation}
  \E A_{ij}^2=\frac{(1+\delta_{ij})\,a}{2}.
\end{equation}
We write $\lambda^{N,a}_1\geq \cdots\geq \lambda_{N}^{N,a}$ for the
ordered eigenvalues of $\GOE_{N}(a)$, and define the averaged empirical
spectral measure by
\begin{equation}
  \label{eq:empirical_measure}
  \mu_{N,a}(A)=N^{-1}\sum_{i=1}^N\P[\lambda_i^{N,a} \in A],\qquad A\in \mathcal B( \R).
\end{equation}
The density of $\mu_{N,a}$ is denoted by $\rho_{N,a}$, and we introduce
$\rho_N := \rho_{N,N^{-1}}$ as a convenient abbreviation. It is
well-known that  $\rho_N(x)$ converges to
$\frac{1}{2\pi }{\sqrt{2-x^2}}\mathbf{1}_{|x|\leq \sqrt{2}}$ as $N\to\infty$,
see e.g.~\cite[(7.2.31)]{mehta2004random}.

%\DB{Shuta todo}
%{\color{cyan} Not that by the semi-cirlce law
%$$\text{(2.6) without uniform}.$$}
We will need the following identity for the determinant of a shifted $\GOE_{N-1}(N^{-1})$.

\begin{lemma}
  \label{lem:detGOE}
  For any $x\in\mathbb{R}$,
  \begin{equation}
    \label{determinant-computation}
    \E\big[\big|\det(x \Id_{N-1}+\GOE_{N-1}(N^{-1}))\big|\big]
      =\sqrt{2}\, N^{-(N-2)/2}\,
      \Gamma\left(\frac{N}{2}\right)e^{Nx^2/2}\rho_{N}(x).
  \end{equation}
\end{lemma}

%\JC{I have feeling that this sentence %is not longer true}
%The following two lemmas appear in %(38) and (49) of
%\cite{fyodorov2013high}, respectively.
We also need precise estimates for $\rho_N(x)$.
\begin{lemma}
  \begin{enumerate}[(i)]
      \item For any $\delta>0$,
      \begin{equation}\label{rhoN asympt}
        \rho_N(x)=
        \frac{\exp{(N\,\Phi(x))}}{2\sqrt{\pi\,N}\,(x^2-2)^{\frac 1 4}
          \,(|x|+\sqrt{x^2-2})^{\frac{1}{2}+o(1)}}
      \end{equation}
      with the error term $o(1)$ converging to zero uniformly
      for $|x|>\sqrt{2}(1+\delta)$, and where
      \begin{equation}
        \label{eq:phi_def}
        \Phi(x)=
        \left(-\frac{|x|\sqrt{x^2-2}}{2}+\ln{\left\{\frac{|x|+\sqrt{x^2-2}}{\sqrt{2}}
        \right\}}\right)\,\1_{\{|x|\geq \sqrt{2}\}} \le 0.
      \end{equation}

      \item
      For any $\e>0$ and large enough $N$, for all $x\in \mathbb R$
      \begin{equation}
        \label{Upper_bound_of_rho_N}
        e^{N\Phi(x)(1+\varepsilon )-N\e}
        \leq \rho_N(x)\leq
        e^{N\Phi(x)(1-\varepsilon )+N\e}.
        %\exp{(N\,(\Phi(x)-\e \max( |\Phi(x)|,1)))} \leq \rho_N(x)\leq
        %\exp{(N\,(\Phi(x)+\e\max( |\Phi(x)|,1)))}.
    \end{equation}
  %  \begin{equation}
  %    \label{Upper_bound_of_rho_N}
  %    \exp{(N\,(\Phi(x)(1+\e)-\e))} \leq \rho_N(x)\leq \exp{(N\,(\Phi(x)(1-\e)+\e))},\quad \forall x\in\R.
  %  \end{equation}
  %
  %    \begin{equation}
  %      \label{Upper_bound_of_rho_N}
  %      \exp{(N\,(\Phi(x)-\e( |\Phi(x)| \wedge 1)))} \leq \rho_N(x)\leq \exp{(N\,(\Phi(x)(1-\e)+\e))},\quad \forall x\in\R.
  %    \end{equation}
  %  \item
  %  For any $\e>0$, compact set $K\subset\mathbb{R}$ with $\{\pm \sqrt{2}\} \not\in K$ and large enough $N$
  %  \begin{equation}
  %    \label{Lower_bound_of_rho_N}
  %     \rho_N(x)\geq \exp{(N\,(\Phi(x)-\e))},\quad \forall x\in K.
  %  \end{equation}}
\end{enumerate}
  \label{lem:rhoapprox}
\end{lemma}

\begin{remark}
  %\DB{Shuta todo: (2.6) with point-wise convergence is the semi-circle law [Mehta?] and uniform convergence on any interval follows since $\rho_N$ are continuous and the limit is continuous.}
  Lemma~\ref{lem:detGOE} is (38) from \cite{fyodorov2013high}.
  For
  Lemma~\ref{lem:rhoapprox}(i) with pointwise convergence see \cite[(49)]{fyodorov2013high} and
  \cite[(3.11)]{forrester2012spectral}. We give self-contained proofs of both lemmas in Appendix A.
  %which upon careful inspection can
  %be seen to also prove the uniform convergence as claimed in (i). %Alternatively the asymptotic behaviour of $\rho_{N}(x)$ integrated over $x$ is
  %obtained in \cite[Theorem 2.17(a) and p.~193]{auffinger2013random} and
  %the same argument is also applicable to obtain (i).
  %The upper bound of  \eqref{Upper_bound_of_rho_N} is
%  an easy consequence of the formula for $\rho_N(x)$ related to the Hermite functions \cite[(7.2.32)]{mehta2004random} and the uniform asymptotics of the Hermite functions \cite[(6.11)--(6.13)]{erdelyi1960asymptotic} and hence, we omit the proof. For the lower bound, we need to estimate the difference of two Hermite functions and the proof is more complicated. Hence, we take a different route, which is presented in the Appendix.
\end{remark}

We record the following easy estimate for $\Phi(x)$ that
  follows directly from \eqref{eq:phi_def}, showing that it grows
  quadratically:
\begin{equation}\label{eq:PhiBound}
  -\frac{x^2}{2} \le \Phi(x)
  \le - \frac{x^2}{2} + c + c'x \text{ for all $x\in\mathbb{R}$, for some constants $c,c'$.}
\end{equation}

\section{Exact formula for the mean number of critical points}
\label{sec:exact}

In this section, we make the first step on the way to prove
Theorems~\ref{thm:main_thm}-\ref{thm:sec_thm}. The main result is
Proposition~\ref{pro:exactN} giving a precise formula for the number of
critical points with certain properties.
%This will be useful to show our
The additional properties will be useful later to show \eqref{eq:main_res_GS_energy}--\eqref{eq:eval_hessian} characterizing the maximizer of $H_N^h$.

To state this proposition we need several definitions.
Given measurable sets $\Gamma\subset[-1,1]$ and  $R,E\subset \R$, we
define
\begin{equation}
  \begin{split}
    \mathcal{N}_N&(\Gamma,R,E)
    \\&=\big|\{\sigma\in S_{N-1} :
      \nabla_\sph H^h_N(\sigma)=0,\,
      \sigma\cdot{\uN} \in \Gamma,\,
      N^{-1}\dN H_N(\sigma) \in R,\,
      N^{-1}H_N^h(\sigma) \in E
    \}\big|
  \end{split}
\end{equation}
(note that the radial derivative $\partial_r H_N(\sigma)$ is indeed of the Hamiltonian $H_N(\sigma)$ without external field, and not of $H^h_N(\sigma)$).
%the intentional absence of the $h$ superscript on $\dN H_N(\sigma)$).
For $\gamma\in (-1,1)$, $x\in \mathbb R$, we set
\begin{align}
  \label{eq:Gdef}
  G(x,\gamma)&=\frac{1}{2}\ln{(1-\gamma^2)}+
  \frac{h^2\gamma^2}{2\xipone}-\frac{\,x^2}{2(\xipone+\xippone)}+
  \frac{\left(x + h \gamma \right)^2}{4\xippone},\\
  \label{eq:pxgamma}
  p_{x,\gamma}(E)&=\P\big(N^{-1} H_N^h(\sigma) \in E\,\big |
    \,N^{-1} \dN H_N(\sigma)=x\big),
\end{align}
where in the last formula $\sigma \in S_{N-1}$ is such that
$\sigma \cdot \uN = \gamma $ (it is easy to see from the symmetry of the
  Hamiltonian that the right-hand side depends on $\sigma$ only through $\gamma$).
Finally, we recall the definition of $\rho_N$ from below
\eqref{eq:empirical_measure}.

\begin{proposition}
  \label{pro:exactN}
  For every $N\in \mathbb N$, measurable $\Gamma \subset [-1,1]$  and
  $R,E\subset \mathbb R$,
  \begin{equation*}
    \begin{split}
      &\E\big[\mathcal N_N(\Gamma ,R, E)\big]\\
      &\quad=
      e^{-\frac{Nh^2}{2\xipone}}
      \left(\frac{\xippone}{\xipone}\right)^{\frac{N-1}{2}}
      \frac{2N}{\sqrt{\pi (\xipone+\xippone)}}
      \,\frac{\Gamma(\frac N2)}{\Gamma(\frac{N-1}{2})}
      \int_\Gamma \int_R
      \frac{e^{NG(x,\gamma)}}{(1-\gamma^2)^{{3}/{2}}}\,
      \rho_N\left(\frac{x+h
          \gamma}{\sqrt{2\xippone}} \right)p_{x,\gamma}(E) \dd x
      \dd \gamma.
    \end{split}
  \end{equation*}
\end{proposition}

\begin{proof}
  By the Kac-Rice formula (see, e.g.~\cite[(12.1.4)]{adler2009random})
  \begin{equation}
    \begin{split}
      \label{eq:KacRiceNN}
      \E\big[\mathcal{N}_N(\Gamma,R,E)\big]
      = \int_{S_{N-1}}
      \E\big[|
        \det \nabla^2_\sph H^h_N(\sigma)| \1_{\kE_{E,R}}
        \, \big| \, \nabla_\sph
        H^h_N(\sigma)=0\big]
      \1_\Gamma ( \gamma )
      \,f_{\nabla_\sph H^h_N(\sigma)}(0)\,
      \dd \sigma,
    \end{split}
  \end{equation}
  where we set
  \begin{equation}
    \label{eq:gamma}
    \gamma  = \gamma (\sigma ) = \sigma  \cdot \mathbf u_N,
  \end{equation}
  and
  \begin{equation}
    \kE_{E,R}=\{N^{-1} H_N^h(\sigma) \in E,\,N^{-1}\dN
      H_N(\sigma) \in R\}.
  \end{equation}
  Using \eqref{eq:Hh}, formulas \eqref{eq:covgrad},
  \eqref{eq:covHess} and
  the notation introduced above them, it follows that
  \begin{align}
    \label{nabla-determinant}
    \nabla_\sph H_N^{h}(\sigma)
    %=(\partial_i H_N^{h}(\sigma))_{i=1}^{N-1}
    &=\big(\partial_i H_N(\sigma)+hN\mathbf{u}_{N}\cdot\mathbf{e}_i(\sigma)\big)_{i=1}^{N-1},
    \\
    \label{nabla2-determinant}
    \nabla^2_\sph H_N^{h}(\sigma)
    &= \nabla^2 H_N^{h}(\sigma)\big|_\sph-\dN H_N^{h}(\sigma)\, \Id_{N-1}
    =
    \nabla^2  H_N(\sigma)\big|_\sph
    -\big(\dN H_N(\sigma)+Nh \gamma  \big)  \Id_{N-1}.
  \end{align}

  The vector that lists all entries of $H_N(\sigma ), \nabla H_N(\sigma )$
  and $\nabla^2 H_N(\sigma )$ is a centred multivariate Gaussian vector.
  Its covariance can be computed from \eqref{eq:covariance_H_N}. This
  computation is standard in the context of the critical point complexity
  for spherical Hamiltonians (see \cite[Lemma~1]{AB13} and
    \cite[Appendix~A]{arous2020geometry}); we recall these results in
  Lemma~\ref{lem:covariances_of_HN} in Appendix B. Here we only need
  the following claims that are a direct consequence of this lemma.

  \begin{lemma}
    \label{lem:law_of_grads}
    For every $\sigma \in S_{N-1}$
    \begin{enumerate}[(a)]
      \item $\nabla_\sph H(\sigma)$ is independent of
      $\big(H_N(\sigma), \dN H_N(\sigma), \nabla^2 H_N(\sigma)|_\sph\big)$,
      and $\nabla^2 H_N(\sigma)|_\sph$ is independent of
      $\big(H_N(\sigma ),\dN H_N(\sigma)\big)$.

      \item  $\partial_i H_N(\sigma),i=1,\ldots,N-1$, are i.i.d.~centred normal
      random variables with  variance $\xipone N$.

      \item  $\nabla^2 H_N(\sigma)|_\sph$ has the law of
      $\GOE_{N-1}(2 \xippone N)\overset{d}{=}N\sqrt{2 \xippone} \, \GOE_{N-1}(N^{-1})$.

      \item $\dN H_N(\sigma)$ is a centred normal random variable with
      variance $(\xipone+\xippone) N$.
    \end{enumerate}
  \end{lemma}

  We now evaluate the terms appearing on the right-hand side of the
  Kac-Rice formula \eqref{eq:KacRiceNN}.

  \begin{lemma}
    \label{lem:density}
    For every $\sigma \in S_{N-1}$, using the notation \eqref{eq:gamma},
    \begin{equation*}
      f_{\nabla_\sph H^{h}_N(\sigma)}(0)= C_1(N)
      \exp{\left(\frac{Nh^2\gamma ^2}{2\xipone} \right)},
    \end{equation*}
    where
    \begin{equation*}
       C_1(N) =
        (2\pi\,\xipone\,N)^{-\frac{N
            -1}{2}}\,\exp{\left(-\frac{Nh^2}{2\xipone}
            \right)}.
      \end{equation*}
  \end{lemma}

  \begin{proof}
    By \eqref{nabla-determinant} and Lemma~\ref{lem:law_of_grads}(b),
    $\nabla_\sph H^{{h}}_N(\sigma )$ is a
    Gaussian vector whose components are independent and whose $i$-th
    component has mean $hN  \uN\cdot \mathbf{e}_i(\sigma)$
    and variance $\xipone N$.  Therefore,
    \begin{equation*}
      f_{\nabla_\sph
        H^h_N(\sigma)}(0)
      =(2\pi\,\xipone\,N)^{-\frac{N- 1}{2}}
      \,\exp{\left(-\frac{Nh^2}{2\xipone}\sum_{i=1}^{N-1}
          (\uN\cdot \mathbf{e}_i(\sigma))^2\right)}.
    \end{equation*}
    Using that
    $1=|\uN|^2=\sum_{i=1}^{N}(\uN\cdot \mathbf{e}_i(\sigma))^2$
    and recalling that $\mathbf{e}_N(\sigma ) = \sigma $ completes the proof.
  \end{proof}

  \begin{lemma}
    \label{lem:determinant}
    For every $\sigma \in S_{N-1}$, every $\Gamma$, $R$, $E$ as in
    Proposition \ref{pro:exactN}, and $h\geq 0$,
    \begin{equation*}
      \label{eq:determinant}
      \begin{split}
        &\E\big[ \lvert \det \nabla^2_\sph H^h_N(\sigma)\rvert \1_{\kE_{E,R}}
          \,\big|\,
          \nabla_\sph H^h_N(\sigma)=0\big]
        \\& = C_2 (N) \int_{R}
        \exp\left(N\left(-\frac{x^{2}}{2(\xipone+\xippone)}
              +\frac{(x+h \gamma )^2}{4\xippone}    \right) \right)
        \rho_N\left(\frac{x+h \gamma }{\sqrt{2\xippone}}\right)
        p_{x,\gamma}(E)
        \dd x,
      \end{split}
    \end{equation*}
    where
    \begin{equation*}
      C_2(N) = \frac{(2\xippone)^{\frac{N-1}{2}}\,N^{\frac{N+1}{2}}\,\Gamma(N/2)}
      {\sqrt{\pi(\xipone+\xippone)}}.
    \end{equation*}
  \end{lemma}

  \begin{proof}
    Using \eqref{nabla-determinant} and \eqref{nabla2-determinant}
    together with Lemma~\ref{lem:law_of_grads}(a), we see that
    $\nabla^2_\sph H_N^h(\sigma )$ and $\mathcal E_{E,R}$ are independent
    of $\nabla_\sph H_N^h(\sigma )$, and therefore we can remove the
    conditioning on $\nabla_\sph H_N^h(\sigma ) = 0$. By
    \eqref{nabla2-determinant} and Lemma~\ref{lem:law_of_grads}(a,c), we
    then obtain
    \begin{equation*}
      \begin{split}
        \E\big[&\big\lvert
          \det \nabla^2_\sph H^h_N(\sigma)\big\rvert \1_{\kE_{E,R}} \big] =
        \E\left[\big\lvert\det \left(  \nabla^2
            H_N(\sigma)|_\sph -\dN
            H_N^{h}(\sigma)\Id_{N-1}\right)\big\rvert \1_{\kE_{E,R}}\right]                 \\
        & =\E\left[\Big\lvert
          \det \left(N \sqrt{2\xippone}\,
            \GOE_{N-1}(N^{-1})-\dN H^h_N(\sigma)\Id_{N-1}
        \right)\Big\rvert \1_{\kE_{E,R}}\right]      \\
        & = (2\xippone N^2)^{\frac{N-1}{2}}\,\E\bigg[\bigg\lvert\det
          \Big(\GOE_{N-1}(N^{-1})
            -\Big(\frac{\dN H_N(\sigma)+Nh\gamma }{N\sqrt{2\xippone}}\Big)\Id_{N-1}
             \bigg)\bigg\rvert \1_{\kE_{E,R}}\bigg]
        \label{XN_appears},
      \end{split}
    \end{equation*}
    where the matrix ${\rm GOE}_{N-1}(N^{-1})$ is independent of
    $\dN H_N(\sigma )$ and $H_N(\sigma )$. Recalling the distribution of
    $\dN H_N(\sigma )$ from Lemma~\ref{lem:law_of_grads}(d), using
    the notation from \eqref{eq:pxgamma}
    to write the expectation as an integral over the value $x$ of
    $N^{-1}\dN H_N(\sigma )$, this becomes
    \begin{equation*}
      \begin{split}
        (2\xippone N^2)^{\frac{N-1}{2}}&\,\sqrt{\frac{N}{2\pi(\xipone+\xippone)}}
        \int_{R}
        \exp{\left(-\frac{Nx^{2}}{2(\xipone+\xippone)}\right)}
        \\&\times
        \E\bigg[\bigg\lvert\det \bigg(\GOE_{N-1}(N^{-1})-
            \frac{x+h \gamma }{\sqrt{2\xippone}} \Id_{N-1}
        \bigg)\bigg\rvert \bigg]
        p_{x,\gamma}(E) \dd x.
      \end{split}
    \end{equation*}
    Lemma~\ref{lem:detGOE} and $\rho_N(x)=\rho_N(-x)$ then yield the claim.
  \end{proof}

  Going back to \eqref{eq:KacRiceNN}, using Lemmas~\ref{lem:density} and
  \ref{lem:determinant},  we obtain
  \begin{equation}
    \label{MAtCH}
    \begin{split}
      & \E\big[\mathcal{N}_N(\Gamma,R,E)\big]
      = C_1(N) C_2(N)
         \int_{  S_{N-1}}\int_{R}  \1_\Gamma(\gamma )
         \exp\bigg(\frac{Nh^2 \gamma^2}{2\xipone}\bigg)
         \\ &\qquad \qquad \times
         \exp{\bigg(N\bigg(-\frac{x^{2}}{2(\xipone+\xippone)}+\frac{(x+h
                \gamma)^2}{4\xippone}\bigg) \bigg)}
      \rho_N\bigg(\frac{x+h \gamma
      }{\sqrt{2\xippone}}\bigg)p_{x,\gamma }(E)
      \dd x \dd \sigma.
    \end{split}
  \end{equation}
  We proceed with the evaluation of the double integral on the right-hand
  side of \eqref{MAtCH} which we denote by $I_N(\Gamma , R, E)$.
  Observing that the integrand depends on $\sigma $ only through $\gamma$,
  we obtain
  \begin{equation}
    \begin{split}
      I_N(\Gamma , R, E) & = \int_{\Gamma}\int_R
      \frac{\Vol(\sqrt{1-\gamma^2}
          S_{N-2})}{\sqrt{1-\gamma^2}} \exp{\bigg(N\bigg(
            \frac{h^2\gamma^2}{2\xipone}-\frac{x^{2}}{2(\xipone+\xippone)}
            +\frac{(x+h \gamma)^2 }{4\xippone}\bigg)\bigg)}        \\
      & \qquad\qquad\times \rho_N\bigg(\frac{x+h
          \gamma}{\sqrt{2\xippone}}\bigg) p_{x,\gamma}(E) \dd x \dd \gamma .
    \end{split}
  \end{equation}
  Using that $ \Vol(r S_{N-2})=2r^{N-2}
    {\pi^{(N-1)/{2}}}/{\Gamma(\frac{N-1}{2})}$ and
  recalling the notation from \eqref{eq:Gdef}, we get
  \begin{equation*}
    \begin{split}
      I_N(\Gamma ,R, E) &=
      \frac{2 \pi^{\frac{N-1}{2}}}{\Gamma(\frac{N-1}{2})}
      \int_{\Gamma}\int_R
      (1-\gamma^2)^{\frac{N-3}{2}}
      \exp{\bigg(N\bigg( \frac{h^2\gamma^2}{2\xipone}-\frac{x^{2}}{2(\xipone+\xippone)}
            +\frac{(x+h\gamma )^2}{4\xippone}\bigg)\bigg)}\\
      &\qquad\qquad\times\rho_N\bigg(\frac{x+h
          \gamma}{\sqrt{2\xippone}}\bigg) p_{x,\gamma}(E)
      \dd x \dd \gamma \\
      &=  \frac{2  \pi^{\frac{N-1}{2}}}{\Gamma(\frac{N-1}{2})}
      \int_\Gamma \int_R
      (1-\gamma^2)^{-\frac{3}{2}}
      e^{NG(x,\gamma)}\rho_N\bigg(\frac{x+h
          \gamma}{\sqrt{2\xippone}} \bigg) p_{x,\gamma}(E) \dd x\dd \gamma .
    \end{split}
  \end{equation*}
  Inserting this into \eqref{MAtCH} and simplifying the prefactors yields
  the claim of Proposition \ref{pro:exactN}.
\end{proof}

\section{Optimising the integrand}\label{sec:opti}

The asymptotic behaviour of $\mathbb E[\mathcal N_N(\Gamma ,R, E)]$ will
be determined using the Laplace method. To this end, we need to control the
exponential growth rate of the integrand in Proposition~\ref{pro:exactN}.
We will discuss the rate of $p_{x,\gamma }(E)$ in Section
\ref{sec:maxchar} and set for this section $E=\R$, so that
$p_{x,\gamma }(E)=1$.

Let
\begin{equation}
  \label{eq:defF}
  \begin{split}
      F(x,\gamma) & =  G(x,\gamma)+\Phi\bigg(\frac{x+\gamma
          h}{\sqrt{2\xippone}}\bigg)                            \\
      & =
      \frac{1}{2}\ln{(1-\gamma^{2})}+\frac{h^{2}\gamma^{2}}{2\xipone}
      -\frac{ x^{2}}{2(\xipone+\xippone)}
      +\frac{(x+h\gamma)^{2}}{4\xippone}
      +\Phi\bigg(\frac{x+h \gamma }{\sqrt{2\xippone}}\bigg),
     % \\
     % & =
     % \frac{1}{2}\ln{(1-\gamma^{2})}+\frac{h^{2}}{2\xipone}\gamma^{2}-\frac{
     %   x^{2}}{2(\xipone+\xippone)}   +\Omega\left(\frac{x+h \gamma
     % }{\sqrt{2\xippone}}\right).
    \end{split}
\end{equation}
with $\Phi$ as in \eqref{eq:phi_def}. The next lemma gives an estimate of
the integrand in Proposition~\ref{pro:exactN} in terms of $F$. Before
stating and proving it, note that \eqref{eq:PhiBound} implies the
following uniform bound for $F(x,\gamma)$:
\begin{equation}\label{ineq: tailest}
  F(x,\gamma)
  \le \frac{1}{2}\ln( 1- \gamma^2) + c_1 - c_2 x^2
  \text{ for some } c_1,c_2 \text{ (depending on $h$, $\xi$) and all }x,\gamma.
\end{equation}

\begin{lemma}
  \label{lem:expass}
  It holds that
  \begin{equation}
    \label{eq:GandF}
    (1-\gamma^2)^{-\frac{3}{2}}  e^{NG(x,\gamma)}
    \rho_N\bigg(
      \frac{x + h \gamma}{\sqrt{2 \xippone}} \bigg)
    = e^{N F(x,\gamma) (1+o(1))+ o(N)},
  \end{equation}
  where the error terms may depend on $h,\xi$, but are uniform in
  $(x,\gamma)\in\R\times(-1,1)$.
\end{lemma}

\begin{proof}
  Using Lemma~\ref{lem:rhoapprox}(ii) one obtains that the logarithm of the
  left-hand side equals
  \begin{equation*}
    N F(x,\gamma) + o\left(N\left(|\ln(1-\gamma^2)|
        + \Big|\Phi\Big(\tfrac{x+h\gamma}{\sqrt{2\xippone}}\Big)\Big| + 1\right)\right).
  \end{equation*}
  On any compact subset of $\R\times(-1,1)$ the error term can be bounded
  by $o(N)$. On the other hand for large enough $x$ the inequalities
  \eqref{eq:PhiBound} and \eqref{ineq: tailest} imply that the error term
  can be bounded by $o(N|F(x,\gamma)|)$.
\end{proof}

%Using Lemma~\ref{lem:rhoapprox} (ii), the integrand of
%Proposition~\ref{pro:exactN} then reads
%\begin{equation}
%  \label{eq:GandF}
%  (1-\gamma^2)^{-\frac{3}{2}}  e^{NG(x,\gamma)}
%  \rho_N\bigg(
%    \frac{x + h \gamma}{\sqrt{2 \xippone}} \bigg)
% = e^{N F(x,\gamma) (1+o(N))+ o(N)},
%\end{equation}
%uniformly for $(x,\gamma)\in\R\times(-1,1)$}, %where with $\Phi$ as in \eqref{eq:phi_def},
%\begin{equation}
  %\label{eq:defF}
  %\begin{split}
%      F(x,\gamma) & =  %G(x,\gamma)+\Phi\bigg(\frac{x+\gamma
%          h}{\sqrt{2\xippone}}\bigg)                            \\
%      & =
%      \frac{1}{2}\ln{(1-\gamma^{2})}+\frac{h^{2}\gamma^{2}}{2\xipone}
%      -\frac{ x^{2}}{2(\xipone+\xippone)}
%      +\frac{(x+h\gamma)^{2}}{4\xippone}
%      +\Phi\bigg(\frac{x+h \gamma }{\sqrt{2\xippone}}\bigg).
%     % \\
     % & =
     % \frac{1}{2}\ln{(1-\gamma^{2})}+\frac{h^{2}}{2\xipone}\gamma^{2}-\frac{
     %   x^{2}}{2(\xipone+\xippone)}   +\Omega\left(\frac{x+h \gamma
     % }{\sqrt{2\xippone}}\right).
%    \end{split}
%\end{equation}
%{\color{green}Note that the somewhat delicate uniformity is due to $F(x,\gamma)\leq c_1-c_2 x^2$ (which we will derive in \eqref{ineq: tailest} later). This gives that $F(x,\gamma)o(1)$ contains $\Phi(\frac{x+\gamma h}{\sqrt{2\xippone}})o(1)$ as the latter approaches $-\infty$ quadratically for large $x$ as well. }

In order to determine the maximum of $F$, it is convenient to make
a number of changes of variables. We eliminate $x$ by setting
\begin{equation}
  \label{eq:defeta}
  \eta=\frac{x+h\gamma}{\sqrt{2\xippone}},
\end{equation}
and introduce
\begin{equation}
  \label{eq:defhtildea}
  \widetilde{h}=\frac{h}{\sqrt{2\xippone}}\in [0,\infty)
  \qquad\text{and}\qquad
  a=\frac{\xipone}{\xippone}\in(0,\infty).
\end{equation}
%Recalling $h_c$ from \eqref{eq:hc}, we note that when $\xippone \ge \xipone$
%\begin{equation}
%  \label{eq:htildehc}
%  h>h_{c}\iff\widetilde{h}>\frac{\sqrt{1-a}}{\sqrt{2}}.
%\end{equation}
We then set
\begin{equation}
  \label{eq:tildeF}
  \widetilde{F}(\eta,\gamma)=F\big(\eta\sqrt{2\xippone}
    -h\gamma,\gamma\big) =
    \frac{1}{2}\ln\big(1-\gamma^2\big)+\frac{\widetilde{h}^2\gamma^2}{a}-
    \frac{(\eta-\widetilde{h}\gamma)^{2}}{1+a}+\frac{\eta^2}{2}+
    \Phi (\eta).
\end{equation}

Note that the maximum values of $F$ and $\widetilde{F}$ over
$\mathbb{R}\times(-1,1)$ coincide and $(x,\gamma)$ is a maximizer of $F$
if and only if $(\eta,\gamma)$ (with $\eta$ as in \eqref{eq:defeta})
is a maximizer of $\widetilde{F}$.

\begin{proposition}
  \begin{enumerate}[(i)]
    \item If $h^2  >  \xippone-\xipone$ (i.e.~if $h>h_c$ or
      $a>1$), then the unique maximizers of $F$ and $\widetilde F$ are
    $\pm (x_*, \gamma_*)$ and $\pm (\eta_*, \gamma_*)$ respectively, where
    \begin{equation}
      \label{eq:etastar}
      \gamma_*=\frac{h}{\sqrt{\xipone+h^2}},
      \qquad \eta_*=\frac{\xipone+\xippone+h^2}
      {\sqrt{2\xippone(\xipone +h^2)}},
      \qquad x_*=\frac{\xipone+\xippone}{\sqrt{\xipone +h^2}}.
    \end{equation}
    %{\color{cyan}If $\xippone \ne \xipone$, then this also holds for $h=h_c$.}
    The common value of their maxima is then
    \begin{equation}
      \label{eq:Fetastar}
      F(x_*,\gamma_*)=\widetilde{F}(\eta_*,\gamma_*)
      =\frac{h^2}{2\xipone}-\frac{1}{2}\ln{
        \frac{\xippone}{\xipone}}.
    \end{equation}

    \item If $a<1$ and $h\in (h_c\sqrt a, h_c]$,
    then the unique maximizers of $F$ and $\widetilde F$ are
    $\pm (x_0, \gamma_0)$ and $\pm (\eta_0, \gamma_0)$ respectively, where
    \begin{equation}
      \label{eq:eta0}
      \gamma_0=\frac{1}{h}
      \sqrt{\frac{\xippone h^2-\xipone(\xippone-\xipone)}{\xippone}},
      \qquad
      \eta_0=\frac {h \gamma_0\sqrt{2 \xippone}}{\xippone-\xipone},
      \qquad
      x_0 =\frac {h\gamma_0 (\xipone+\xippone)}{\xippone-\xipone}.
    \end{equation}
    The common value of their maxima is then
    \begin{equation}
      \label{eq:Feta0}
      F(x_0,\gamma_0)=\widetilde{F}(\eta_0,\gamma_0)=\frac{1}{2}(H-1-\ln{H}),
    \end{equation}
    where
    \begin{equation}
      H=\frac{\xippone
        h^2}{\xipone(\xippone-\xipone)}= \frac{\xippone}{\xipone} \frac{h^2}{h_c^2}.
    \end{equation}

    \item If $a<1$ and $h \le h_c\sqrt a$, then $(0,0)$ is the unique maximizer of $F$ and
    $\widetilde F$
    and $F(0,0) = \widetilde F(0,0) = 0$.
    \item
      If $a=1$ and $h=0$, then
      $F(x,\gamma)=\tilde{F}(\eta,\gamma)=\frac{1}{2}\ln{(1-\gamma^2)}+\Phi(\eta)$
      with $\eta=\frac{x}{\sqrt{2\xippone}}$ and the maximum of $F$ over
      $\R \times (-1,1)$ is $0$.
  \end{enumerate}
  \label{pro:maxF}
\end{proposition}

\begin{proof}
Note that from \eqref{eq:phi_def}
  \begin{equation}\label{eq: phi deriv}
   \Phi'(\eta) = -\sqrt{\eta^{2}-2}\,\sgn(\eta)
      \,\1_{ |\eta|\ge\sqrt{2} },
  \end{equation}
so that $\Phi$ is a differentiable function (including at $\pm\sqrt{2}$). Thus
  $\widetilde{F}$ is differentiable as well.
  In addition, from \eqref{ineq: tailest} it follows that $\widetilde{F}(\eta,\gamma)$ tends to $-\infty$ as
  $|\gamma|\to 1$ or $|\eta|\to\infty$. Therefore, a maximizer of $\widetilde F$
  must exist and be a critical point. Moreover, for every  $\eta\ge 0$ and
  $\gamma \ge 0$,
  \begin{equation}
    \label{eq:Fsym}
    \widetilde{F}(\eta,\gamma)=\widetilde{F}(-\eta,-\gamma)
    \geq \widetilde{F}(-\eta,\gamma)=\widetilde{F}(\eta,-\gamma),
  \end{equation}
  where equality holds only when $\eta=\gamma=0$.  Hence, to look for a
  maximizer, we only need to consider the critical points in
  $[0,\infty)\times[0,1)$. Thus we assume $\eta,\gamma\ge0$
    in the remainder of the proof. Taking the derivatives in $\eta$ and $\gamma$, and using \eqref{eq: phi deriv}
  we obtain
  \begin{align*}
    \partial_{\eta}\widetilde{F}(\eta,\gamma)
    &=
    %\frac{2\widetilde{h}}{1+a}
    A\gamma-
    %\frac{1-a}{1+a}
    B\eta-\sqrt{\eta^{2}-2}\,
    \1_{ |\eta|\ge\sqrt{2}},\\
    \partial_{\gamma}\widetilde{F}(\eta,\gamma)
    &=C\gamma+A\eta-\frac{\gamma}{ 1-\gamma^{2}},
  \end{align*}
  with
  \begin{equation}
    A=\frac{2\widetilde{h}}{1+a}, \qquad B=\frac{1-a}{1+a}\qquad\text{and}\qquad
    C=\frac{2\widetilde{h}^{2}}{(1+a)a}.
    \label{eq:ABC}
  \end{equation}
  Hence, the critical points of $\widetilde F$ in $[0,\infty)\times(0,1)$ solve the system
  \begin{equation}
    \label{eq:system1}
    \begin{split}
      A\gamma-B\eta & =  \sqrt{\eta^{2}-2}
      \,\1_{ |\eta|\ge\sqrt{2} },          \\
      C\gamma+A\eta & =  \frac{\gamma}{1-\gamma^{2}}.
    \end{split}
  \end{equation}

  Solutions of this system are illustrated on
  Figure~\ref{fig:tildeFcrit}. The next two lemmas give  the
  solutions in various regimes. For the first one recall the definitions
  of $\gamma_*$, $\eta_*$ from \eqref{eq:etastar}.

  \begin{figure}[ht]
    \centering
    \includegraphics[height=5.7cm]{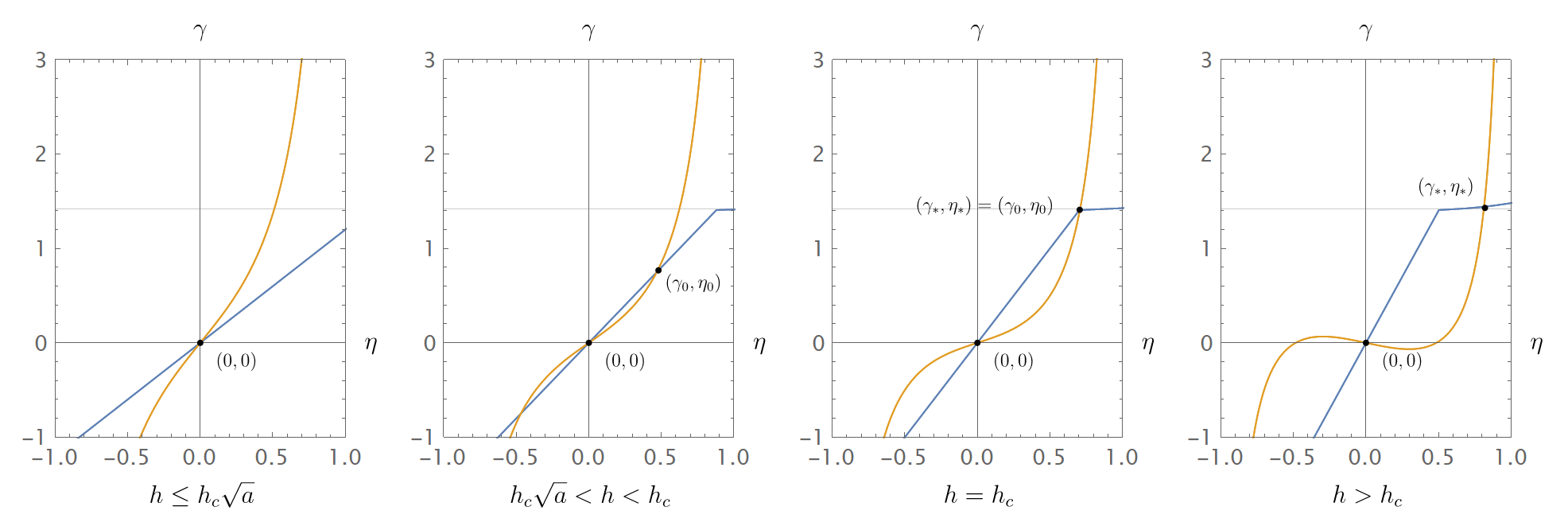}
    \caption{Points $(\gamma ,\eta )$ solving the first (blue line) and
      the second (yellow line) equation of the system \eqref{eq:system1}
      for pure $3$-spin Hamiltonian.}
    \label{fig:tildeFcrit}
  \end{figure}

  \begin{lemma}
    \label{lem:system_with_rhs}
    If $h^2\ge \xippone-\xipone$ (i.e. $a<1$ or $h \ge h_c$) then the point
    $(\eta_*, \gamma_*)$ is the only solution to the system
    \eqref{eq:system1} (and thus the only critical point of $\widetilde F$)
    in $[\sqrt{2},\infty)\times[0,1)$.

    If $h^2< \xippone-\xipone$ (i.e. $a>1$ and $h<h_c$) there is no solution to
    \eqref{eq:system1} in $[\sqrt{2},\infty)\times[0,1)$.
  \end{lemma}

  \begin{proof}
    We first consider the case $h=0$. Then the calculation reduces to the one for purely random external field carried out in \cite{fyodorov2013high}. Indeed $A=C=0$, and the second
    equation of \eqref{eq:system1} implies $\gamma=0=\gamma_*$.  By the
    first equation, $\eta=\sqrt{{2}/({1-B^2})}$ and since
    \begin{equation}\label{eq:Bsquared ident}
    1-B^2= \frac{4a}{(1+a)^2},
    \end{equation}we have $\sqrt{{2}/({1-B^2})} = \eta_*$. Therefore,
      $(\eta_*,\gamma_*)$ is the unique solution of \eqref{eq:system1},
    and the proof is completed in this case.

    From now on we assume $h>0$.   For $\eta\geq \sqrt{2}$, the first
    equation of \eqref{eq:system1} implies that
    $\eta^{2}-2=\left(B\eta-A\gamma\right)^{2}$, which yields the equation
    \begin{equation*}
      (1-B^{2})\eta^{2}+2AB\eta\gamma-(2+A^{2}\gamma^{2}) = 0.
    \end{equation*}
    As $B\in (-1,1)$ this is a quadratic equation for $\eta$ with two solutions, only one of which is non-negative since $2(1-B^{2}) + A^{2}\gamma^{2} > A^2 B^2 \gamma^2 $, namely
    \begin{equation*}
      \eta=\frac{-AB\gamma + \sqrt{2(1-B^{2}) + A^{2}\gamma^{2}}}{1-B^{2}}.
    \end{equation*}
    
    We can rewrite this as
    \begin{equation}
      \eta=\frac{1}{R}\bigg(-\frac{B}{A}\gamma +
        \frac{1}{A}\sqrt{2R+\gamma^{2}}\bigg),\label{eq:eta_from_gamma-1}
    \end{equation}
    with %{\color{red} Shuta todo: adapt to include case $h=0$}
    \begin{equation}
      \label{eq:R}
      R=\frac{1-B^{2}}{A^{2}}\overset{\eqref{eq:Bsquared ident},\eqref{eq:ABC}}=
      \frac{a}{\widetilde{h}^{2}}=\frac{2\xipone}{h^{2
      }}.
    \end{equation}
    Plugging this into the second equation of \eqref{eq:system1}, we
    get
    \begin{equation*}
      C\gamma+\frac{A}{R}\left(-\frac{B}{A}\gamma +
        \frac{1}{A}\sqrt{2R+\gamma^{2}}\right)=\frac{\gamma}{1-\gamma^{2}}.
    \end{equation*}
    Observing that $\gamma=0$ is not a solution to this equation, after dividing through by
    $\gamma$ and some simplifications, we obtain
    %\[
    %  C+A\frac{1}{R}\left(-\frac{B}{A} +
    %    \frac{1}{A}\sqrt{\frac{2R}{\gamma^{2}}+1}\right)=\frac{1}{1-\gamma^{2}},
    %\]
    %which can be writen as
    \begin{equation}
      \label{eq:BLA-1}
      CR-B + \sqrt{\frac{2R}{\gamma^{2}}+1}=\frac{R}{1-\gamma^{2}}.
    \end{equation}
    By \eqref{eq:ABC}, $CR-B=1$. Hence, this is equivalent to
    \begin{equation}\label{eq: R equation}
      \sqrt{1+\frac{2R}{\gamma^2}}=\frac{R}{1-\gamma^2}-1.
    \end{equation}
    If one squares both sides to eliminate the square root and multiplies
    out by $\gamma^2(1-\gamma^2)^2$ to remove all fractions, then one
    gets what is a priori a third degree equation in $\gamma^2$. However
    the third and second degrees cancel yielding the equation $2R(1-2\gamma^2) = R^2 \gamma^2 - 2R\gamma^2$, which has a single non-negative  solution
    \begin{equation}
      \label{eq:gamma_sol_eta_large}
      \gamma=\sqrt{\frac{2}{R+2}}
      \overset{\eqref{eq:R}} {=}
      \frac{h}{\sqrt{\xipone+h^{2}}}\overset{\eqref{eq:etastar}}=
      \gamma_*.
    \end{equation}
    Plugging this into \eqref{eq:eta_from_gamma-1} and using that $\sqrt{2R + \gamma_*^2} = (h^2 + 2\xipone)/h$ we obtain that
    $\eta$ must satisfy
    \begin{equation}\label{eq:eta_star}
      \eta= \frac{1}{\sqrt{\xipone + h^2}}
      \frac{1}{A R h} \left( -B h^2 + h^2 + 2 \xipone\right)=
      \frac{h^{2}+\xipone+\xippone}{\sqrt{2\xippone(\xipone+h^{2})}}
      \overset{\eqref{eq:etastar}}= \eta_*.
    \end{equation}
    Hence the only possible solution of \eqref{eq:system1} in
    $[\sqrt{2},\infty)\times [0,1)$ is $(\eta_*,\gamma_*)$.

    To check that $(\eta_*,\gamma_*)$ \emph{is} indeed a solution (since we squared the equation \eqref{eq: R equation} we must verify this), we
    firstly note that by the definition of $\eta_*$
    \begin{equation}\label{etastar_at_least_root2}
        \eta_*^2 -2 = \frac{(\xipone+h^2-\xippone)^2}{2\xippone(\xipone+h^2)},
    \end{equation} so
    that obviously $\eta_* \ge 2$.
    Secondly, by plugging $\eta=\eta_*$,
      $\gamma=\gamma_*$ into the second equation
      of \eqref{eq:system1} we see that these always solve the equation, since
      %\begin{align*}
    %    C\gamma_*+A\eta_*&=
     %   \frac{2\tilde{h}}{\xipone\sqrt{2%\xippone(\xipone+h^2)}(1+a)}
        %(\xippone h^2+\xipone(\xippone+\xipone+h^2))\\
        %&=\frac{h\sqrt{\xipone+h^2}}{\xipone}=\frac{\gamma_*}{1-\gamma_*^2}.
      %\end{align*}
    \begin{equation}
        C\gamma_*+A\eta_* =
        \frac{2\tilde{h}}{(1+a)\xipone\sqrt{2\xippone(\xipone+h^2)}}
        (\xippone h^2+\xipone(\xippone+\xipone+h^2)),
    \end{equation}
    where the last parenthesis factors as $(\xipone + h^2)(\xipone + \xippone)=\xippone(\xipone + h^2)(1+a)$, giving
     \begin{equation}        C\gamma_*+A\eta_* =\frac{\tilde{h}\sqrt{2\xippone}\sqrt{\xipone+h^2}}{\xipone}=\frac{h\sqrt{\xipone+h^2}}{\xipone}=\frac{\gamma_*}{1-\gamma_*^2}.
     \end{equation}
      Lastly, the left-hand side of the first equation of
      \eqref{eq:system1} equals
      \begin{align*}
        A\gamma_*-B\eta_*
        %&=\frac{2h^2-(1-a)(\xipone+h^2+\xippone)}{\sqrt{2\xippone(\xipone+h^2)}(1+a)}\\
        &=\frac{2h^2\xippone-(\xippone-\xipone)(\xipone+\xippone+h^2)}{(\xipone+\xippone)\sqrt{2\xippone(\xipone+h^2)}}\\
        &= \frac{2\xippone h^2-h^2(\xippone-\xipone)-(\xippone-\xipone)(\xipone+\xippone)}{\sqrt{2\xippone(\xipone+h^2)}(\xipone+\xippone)}\\
        &=\frac{(h^2+\xipone-\xippone)(\xipone+\xippone)}
        {\sqrt{2\xippone(\xipone+h^2)}(\xipone+\xippone)}
        =\frac{h^2+\xipone-\xippone}{\sqrt{2\xippone(\xipone+h^2)}}.
      \end{align*}
      
      The right-hand side of the first equation of \eqref{eq:system1}
      is $\sqrt{\frac{(\xipone+h^2-\xippone)^2}{2\xippone(\xipone+h^2)}}$ by \eqref{etastar_at_least_root2}, which equals the last expression in the above display as long as $h^2 \ge \xippone - \xipone$.
      %
      %equals
      %\begin{align*}
      %  \sqrt{\eta_*^2-2}
      %  =\sqrt{\frac{(\xipone+\xippone+h^2)^2-4\xippone(\xipone+h^2)}{2\xippone(\xipone+h^2)}}
      %  =\sqrt{\frac{(\xipone+h^2-\xippone)^2}{2\xippone(\xipone+h^2)}}.
      %\end{align*}
    %\SN{using
    %\begin{align*}
    %2\xippone h^2-(\xippone-\xipone)(\xipone+h^2+\xippone)&=2\xippone h^2-h^2(\xippone-\xipone)-(\xippone-\xipone)(\xipone+\xippone)\\
    %&=(\xipone+h^2-\xippone)(\xipone+\xippone) ,
    %\end{align*}}
    %Thus, the first equation simplifies %to
    %\begin{equation*}
    %  h^2+\xipone-\xippone =\sqrt{(h^2+\xipone-\xippone)^2},
    %\end{equation*}
    %which is satisfied iff $h^2\geq \xippone-\xipone$.
  \end{proof}

  We now inspect critical points with $\eta \in [0,\sqrt 2]$. Recall the
  definition of $\eta_0$, $\gamma_0$ from Proposition~\ref{pro:maxF}.

  \begin{lemma}
    \label{lem:system_0}
      $\widetilde F$ has at most two critical points in
      $[0,\sqrt 2]\times [0,1)$:

      If $a < 1$
      and $h \le \sqrt{a}h_c$ or if $h^2> \xippone-\xipone$ (i.e. $a>1$ or $h>h_c$) then  the point $(0,0)$ is the only  solution
      to \eqref{eq:system1} in $[0,\sqrt 2]\times [0,1)$.

      If $a<1$ and $h \in (h_c\sqrt a, h_c]$, then the
      points $(0,0)$ and $(\eta_0, \gamma_0)$ are the only solutions to
      \eqref{eq:system1} in $[0,\sqrt 2]\times [0,1)$.
  \end{lemma}

  \begin{proof}
  As claimed $\gamma=0,\eta=0$ is always a solution to \eqref{eq:system1}. In the remainder of the proof we thus seek to determine when there are other non-negative solutions.

  If $\gamma=0$, then $\eta=0$ is the only solution of \eqref{eq:system1}
    in $[0,\sqrt{2})$, and we can thus assume that $\gamma> 0$.

    We first consider the case $a=1$ (that is $\xippone - \xipone=0$) and $h>0$. Then $B=0$ and $A>0$, and
    the first equation of \eqref{eq:system1} leads to $\gamma=0$, showing that $(0,0)$ is the only solution and completing the proof in this special case.

    From now on we can assume $a\neq 1$, and thus $B\neq 0$.  Since we consider $\eta\in[0,\sqrt{2}]$
    only, the first equation in \eqref{eq:system1} is linear and implies
    $\eta=\frac{A}{B}\gamma$.  If $a>1$, then $A>0$ and $B<0$, and thus
    there is no solution to \eqref{eq:system1} with $\gamma,\eta>0$. This completes the proof in the case $a>1$. Hence, we assume $a<1$ for the rest of the proof.

    Plugging $\eta = \frac AB \gamma$ into the second equation, we obtain
    \begin{equation*}
      \Big(C+\frac{A^{2}}{B}\Big)\gamma=\frac{\gamma}{1-\gamma^{2}}.
    \end{equation*}
    Using the identity $C+\frac{A^2}B=\frac{h^2}{(\xippone-\xipone)a}$
    which follows from \eqref{eq:ABC}, it is easy to see that any non-zero solution satisfies
    \begin{equation}
      \label{eq:gamma_sol}
      \gamma^2 = 1-\frac{(\xippone-\xipone)a}{h^2}.
    \end{equation}
    If $h \le h_c \sqrt{a}$ then the right-hand side is non-positive, so that there are no non-zero solutions to \eqref{eq:gamma_sol} and thus no further solutions to \eqref{eq:system1}. This completes the proof of the case $a<1$ and $h \le h_c \sqrt{a}$.
    When $h > h_c \sqrt{a}$ then \eqref{eq:gamma_sol} has unique positive solution
    \begin{equation}\label{eq: gamma0}
        \sqrt{1-\frac{(\xippone-\xipone)a}{h^2}}\overset{\eqref{eq:defhtildea},\eqref{eq:eta0}}= \gamma_0.
    \end{equation}
    The matching
    $\eta$, computed from the first equation, is given by
    $\eta  = \frac AB \gamma_0 = \eta_0$ (see \eqref{eq:ABC}), for $\eta_0$ as
    claimed in \eqref{eq:eta0}.
    Thus $(\eta_0,\gamma_0)$ is the only possible solution to \eqref{eq:system1} in $[0,\sqrt{2}] \times [0,1)$ other than $(0,0)$, and is a solution if indeed $\eta_0 \le \sqrt{2}$.

    If $h\leq  h_{c}$, then $\eta_0 \le \sqrt 2$ holds true, because $\eta_0$ is an increasing functions of
    $h$ by \eqref{eq: gamma0} and \eqref{eq:eta0}, and $\eta_0 = \sqrt 2$ for $h = h_c$, by
    \eqref{eq:eta0}. This completes the proof of the case $a<1,h \in (\sqrt{a} h_c , h_c]$.

    Otherwise, if $h>h_c$, then
    $\eta_0>\sqrt{2}$, so $(\eta_0,\gamma_0)$ is not a solution to \eqref{eq:system1}. This completes the proof of the case $a<1,h>h_c$.
  \end{proof}

  We now have all ingredients to complete the proof of Proposition~\ref{pro:maxF}.

  \begin{proof}[Proof of (iii)]
    The claim (iii) follows directly from the
    previous two lemmas, as $(0,0)$ is the only critical point for
    $h\le h_c\sqrt {a}$, so it must be
    a maximum.
  \end{proof}

  For claims (i), (ii) we need to evaluate $\widetilde F$ at
  the remaining critical points and show that it is positive there.

  \begin{proof}[Proof of (i)]
  By
    Lemmas~\ref{lem:system_with_rhs},~\ref{lem:system_0}, and
    \eqref{eq:Fsym} the possible maximizers  of $\tilde F$ are $(0,0)$ and $\pm (\eta_*,\gamma_*)$.
    We need to compute
    $\widetilde F (\eta_*, \gamma_*)$ and show that it is positive. To this end we write
    \begin{equation}\label{eta representation}
      \eta_{*}=\frac{1}{\sqrt{2}}\Big(z+\frac{1}{z}\Big)\quad\text{for
      }z=\sqrt{\frac{\xippone}{h^{2}+\xipone}}.
    \end{equation}
    If $x=\frac{1}{\sqrt{2}}(z+\frac{1}{z})$
    for $0<z<1$, then the identity
    \begin{equation}
      \frac{x^{2}}{2}+\Phi(x)=\frac{1}{2}+\frac{z^2}{2}
      -\ln z
      \label{eq:Omega_identity}
    \end{equation}
    can be proved from the definition \eqref{eq:phi_def} of $\Phi$ by noting
    that $\sqrt{x^2-2}=\frac{1}{\sqrt{2}}(\frac{1}{z}-z)$. In
    addition,  by \eqref{eq:etastar},
    \begin{equation*}
      \frac{1}{2}\ln\big(1-\gamma_{*}^{2}\big)=
      \frac{1}{2}\ln\bigg(\frac{\xippone}{\xipone+h^{2}}\bigg).
    \end{equation*}
    Inserting the last two displays into~\eqref{eq:tildeF} and cancelling
    the terms containing the logarithm, we obtain
    \begin{equation}\label{eq:F_tilde}
      \widetilde{F}(\eta_{*},\gamma_{*})
      =\frac{\widetilde{h}^{2}\gamma _{*}^{2}}{a}
      -\frac{(\eta_{*}-\widetilde{h}\gamma_{*})^{2}}{1+a}+\frac{1}{2}+
      \frac{1}{2}\frac{\xippone}{h^{2}+\xipone} - \frac{1}{2}\ln
      \bigg(\frac{\xippone}{\xipone}\bigg).
    \end{equation}
    Plugging in all definitions (see  \eqref{eq:defhtildea} and
      \eqref{eq:etastar}) and using $\frac{\widetilde{h}^{2}\gamma _{*}^{2}}{a}=\frac{h^{4}}{2\xipone\left(\xipone+h^{2}\right)}$ and $\frac{(\eta_{*}-\widetilde{h}\gamma_{*})^{2}}{1+a}=\frac{1}{2}\frac{\xipone+\xippone}{\xipone+h^{2}}$ one
      verifies \eqref{eq:Fetastar}.

    Applying the elementary inequality
      \begin{equation}\label{eq: elemineq}
        \ln y \le y-1 \text{ for all }y>0\quad \text{ (with equality only if }y=1).
      \end{equation}
    with $y=\frac{\xippone}{\xipone}$ to
    \eqref{eq:Fetastar} we get for
    $h^2> \xippone-\xipone$ that
    \begin{equation}\label{eq: etagammaineq}
      \widetilde{F}(\eta_*,\gamma_*)
      >\frac{\xippone-\xipone}{2\xipone}
      -\frac{1}{2} \Big(\frac{\xippone}{\xipone}-1\Big)=0 =
      \widetilde{F}(0,0).
    \end{equation}
    Thus $\pm(\eta_*,\gamma_*)$ are the unique
    maximizers of $\tilde F$ if $h^2> \xippone-\xipone$.
    This proves claim (i) of the proposition.
  \end{proof}

  \begin{proof}[Proof of (ii)]
    By
    Lemmas~\ref{lem:system_with_rhs}, \ref{lem:system_0} and
    \eqref{eq:Fsym},  we have that $\pm(\eta_0,\gamma_0)$ are the only possible
    maximizers of $\tilde F$ for $h\in (h_c\sqrt a, h_c)$.
    When $a<1$ and $h=h_c$, then $\eta_*=\eta_0=\sqrt{2}$, $x_*=x_0$ and
    $\gamma_*=\gamma_0$, so this holds also for $h=h_c$.
    Thus to verify (ii) we must compute $F(x_0,\gamma_0)$. To this end we use
    that $\Phi(\eta_0)=0$ since $|\eta_0|\leq \sqrt{2}$. Plugging
    \eqref{eq:eta0} into the \eqref{eq:tildeF} and
    using that $ 1-\gamma_0^2 = \frac{1}{H}$ we obtain
    \begin{align*}
      &\frac{\widetilde{h}^{2}\gamma_0^2}{a}-
      \frac{(\eta_0-\widetilde{h}\gamma_0)^{2}}{1+a}+\frac{\eta_0^2}{2} \\
      &=\gamma_0^2\left(\frac{h^2}{2 \xippone a}-
        \frac1{1+a}
        {\bigg(\frac{h\sqrt{2\xippone}}{\xippone-\xipone}-\frac{h}{\sqrt{2\xippone}}\bigg)^{2}}
        +\frac{2 h^2  \xippone}{2 (\xippone-\xipone)^2 }\right)\\
      &=\left(1-\frac{1}{H}\right)\left(\frac{h^2}{2 \xipone}-
        \frac{h^2}{\xippone+\xipone}
        \bigg({\frac{2\xippone^2}{(\xippone-\xipone)^2}-2\frac{\xippone}{\xippone-\xipone}}+\frac{1}{2}\bigg)
        +\frac{2 h^2  \xippone}{2 (\xippone-\xipone)^2 }\right)\\
      & = \left(1-\frac{1}{H}\right)\left(\frac{h^2}{2 \xipone}+
          \frac{h^2}{2} \frac{\left(\xippone\right)^2
          -
          \left(\xipone\right)^2}{\left(\xippone - \xipone \right)^2\left(\xippone + \xipone \right)} \right)\\
      &=\left(1-\frac{1}{H}\right)\left(\frac{h^2}{2 \xipone}+
        \frac{ h^2}{2(\xippone-\xipone)}\right) = \left( 1-\frac{1}{H} \right) \frac{H}{2} =  \frac{H-1}{2}.
    \end{align*}
    This gives \eqref{eq:Feta0}.  As $H>1$ for $h>\sqrt{a}h_c$ and $a<1$, we see immediately using \eqref{eq: elemineq} that
    $\widetilde F(\eta_0, \gamma_0)>0  = \widetilde F(0,0)$. This
    proves the claim (ii).

  \end{proof}

  \begin{proof}[Proof of (iv)]
      Substituting $\xippone=\xipone$ and $h=0$ into $F$ and $\tilde{F}$, it is straightforward to check $F(x,\gamma)=\tilde{F}(\eta,\gamma)=\frac{1}{2}\ln{(1-\gamma^2)}+\Phi(\eta).$ By \eqref{eq: phi deriv}, $\Phi(\eta)\leq 0$ for any $\eta\in\R$. Hence, since $\ln{(1-\gamma^2)}\leq 0$, the maxima of $F$ and $\tilde{F}$ are $0$.
  \end{proof}

  This completes the proof of all parts of Proposition~\ref{pro:maxF}.
\end{proof}

Having control over the exponential term we are prepared to finish this
section with the proof of claim \eqref{eq:main_res_triv2} which gives the annealed
complexity in the nontrivial regime.

\begin{proof}[Proof of Theorem \ref{thm:sec_thm}(ii)]
  By Proposition \ref{pro:exactN} with $R=E=\mathbb{R}$ and
  $\Gamma = [-1,1]$,
  \begin{equation*}
    \mathbb{E}[\mathcal{N}_{N}]
    =  \exp\left( N \left[-\frac{h^2}{2\xipone}+\frac{1}{2}\ln\left(\frac{\xippone}{\xipone}\right)\right]+o(N)\right)
    \int_{-1}^{1} \int_{-\infty}^{\infty}
    \frac{e^{NG(x,\gamma)}}{(1-\gamma)^{3/2}}\,
    \rho_N\left(\frac{x+h
        \gamma}{\sqrt{2\xippone}} \right) \dd x
    \dd \gamma,
  \end{equation*}
  which by Lemma~\ref{lem:expass} equals
  \begin{equation*}
    \exp\left( N \left[-\frac{h^2}{2\xipone}+\frac{1}{2}\ln\left(\frac{\xippone}{\xipone}\right)\right]+o(N)\right)
    \int_{-1}^{1} \int_{-\infty}^{\infty}
    e^{N F(x,\gamma)(1+o(1))}\,
    \dd x \dd \gamma .
  \end{equation*}

  Bounding the integral over the complement of a sufficiently large box
  around the origin from above using \eqref{ineq: tailest} gives an upper
  bound of $e^{-L N}$ for any $L$. Hence restricting to a bounded region
  only causes vanishing multiplicative error and then the Laplace method yields
  \begin{equation}\label{eq:upperbound_exp_scale}
    \mathbb{E}\left[\mathcal{N}_{N}\right]
    = \exp\left(N\left[\frac{1}{2}\ln\left(\frac{\xippone}{\xipone}\right)
        -\frac{h^2}{2\xipone}
        +\max_{\gamma\in[-1,1],x\in \mathbb{R}} F(x,\gamma)
    \right]+o(N)\right).
  \end{equation}
   Applying Proposition \ref{pro:maxF}(ii, iii, iv) then implies the claim
   \eqref{eq:main_res_triv2}.
\end{proof}

Note that by the same argument using \eqref{eq:Fetastar} of Proposition \ref{pro:maxF}(i) we obtain
\begin{equation}
  \lim_{N\to\infty}\frac{1}{N}\ln \mathbb{E}\left[\mathcal{N}_{N}\right]=
  0 \quad \text{if
  } h^2>\xippone-\xipone,
\end{equation}
which is a weaker form of the triviality claimed in \eqref{eq:expNtriv}.

\section{Exact asymptotic in the trivial regime}
\label{complexity_triviality}

In this section we consider the trivial regime $h^2 > \xippone - \xipone$, and conclude the proof of the asymptotic
complexity \eqref{eq:expNtriv}.

\begin{proof}[Proof of Theorem~\ref{thm:sec_thm}(i)]
  By Proposition \ref{pro:exactN}, using $E=R=\mathbb{R}$ and
  $\Gamma=[-1,1]$, we have
  \begin{equation*}
    \E[\mathcal N_N]=
    e^{-\frac{Nh^2}{2\xipone}}\left(\frac{\xippone}{\xipone}\right)^{\frac{N-1}{2}}
    \,\frac{2N}{\sqrt{\pi
        (\xipone+\xippone)}}\,\frac{\Gamma(\frac N2)}{\Gamma(\frac{N-1}{2})}
    \int_\R\int_{[-1,1]} \frac{e^{NG(x,\gamma)}}{(1-\gamma^2)^{{3}/{2}}}
    \rho_N\left(\frac{x+h
        \gamma}{\sqrt{2\xippone}} \right)
    \dd \gamma \dd x  .
  \end{equation*}
  Following the argument of the proof of Theorem \ref{thm:sec_thm}(ii) and using \eqref{ineq: tailest} we can bound the integral outside a sufficiently large box above by $e^{-L N}$ for any $L$. Hence removing the complement of a sufficiently large box only causes vanishing error. By Laplace principle we then may further restrict to any fixed neighborhood $[ x_* - \varepsilon, x_* + \varepsilon] \times [\gamma_* - \varepsilon, \gamma_*+\varepsilon]$ of
  the maximizers  of the exponential contribution, still causing only vanishing error. Since we assume that
  $h^2>\xippone-\xipone$, it follows from
  Proposition \ref{pro:maxF}(i) that these maximizers are $\gamma_* $ and
  $x_*$, and in addition $ \frac{x_*+h\gamma_*}{\sqrt{2\xippone}}=\eta_*>\sqrt{2}$ by \eqref{etastar_at_least_root2}.
  Hence choosing $\varepsilon>0$ small enough allows the use of Lemma \ref{lem:rhoapprox}(i). Recalling the
  definition of $\eta $ from \eqref{eq:defeta}, we obtain that $\E[\mathcal N_N]$ equals
  \begin{equation*}
    e^{-\frac{Nh^2}{2\xipone}}\left(\frac{\xippone}{\xipone}\right)^{\frac{N-1}{2}}
    \frac{2\sqrt{N}}{\pi\sqrt{
        (\xipone+\xippone)}}\,\frac{\Gamma(\frac N2)}{\Gamma(\frac{N-1}{2})}
    \int_{x_* - \varepsilon}^{x_* + \varepsilon}
    \int_{\gamma_* - \varepsilon}^{\gamma_* + \varepsilon}
    \frac{(1-\gamma^2)^{-\frac{3}{2}}e^{N F(x,\gamma)}}
    {(\eta^2-2)^{\frac 14}\,(|\eta|+\sqrt{\eta^2-2})^{\frac 12}}
    \dd \gamma \dd x + o(1).
  \end{equation*}
  Note that an extra factor $2$ arises since the neighborhood
  of $(-\gamma_*,-x_*)$ by symmetry has the same contribution as the
  neighborhood of $(\gamma_*,x_*)$. Using the Laplace principle with second
  order corrections we obtain that the above is
  \begin{equation*}
    e^{-\frac{N h^2}{2\xipone}}
    \left(\frac{\xippone}{\xipone}\right)^{\frac{N-1}{2}}
    \frac{2\sqrt{N}}{\pi\sqrt{
        (\xipone+\xippone)}}\,\frac{\Gamma(\frac N2)}{\Gamma(\frac{N-1}{2})}
    \frac{(1-\gamma_*^2)^{-\frac{3}{2}} e^{N F(x_*,\gamma_*)} }
    {(\eta_*^2-2)^{\frac 14}\,(|\eta_*|+\sqrt{\eta_*^2-2})^{\frac 12}}
    \frac{2\pi}{N\sqrt{|\operatorname{det} \nabla^2 F(x_*,\gamma_*)|}} + o(1).
  \end{equation*}
  Using that
  ${\Gamma(\frac N2)}/{\Gamma(\frac{N-1}{2})}\sim \sqrt{N/2}$ and plugging in the value of
  $F(x_*,\gamma_*)$ from \eqref{eq:Fetastar} of
    %= \frac{h^2}{2\xipone}-\frac{1}{2}\ln{ \frac{\xippone}{\xipone}}$
  Proposition~\ref{pro:maxF}(i), we benefit from cancellations
  and obtain
  \begin{equation}\label{eq:chap5comp1}
    \E[\mathcal N_N]= \frac{2\sqrt{2}\sqrt{\xipone}(1-\gamma_*^2)^{-\frac{3}{2}}}
    {\sqrt{\xippone}\sqrt{
        (\xipone+\xippone)}(\eta_*^2-2)^{\frac 14}
      \,(|\eta_*|+\sqrt{\eta_*^2-2})^{\frac 12}\sqrt{|\operatorname{det} \nabla^2 F(x_*,\gamma_*)|}}+o(1).
  \end{equation}
  Calculating the second order
  derivatives of $F$ from \eqref{eq:defF}, we obtain:
  \begin{align}
     \partial^2_x
    F(x,\gamma)&=\frac{1}{2\xippone}\left(\frac{\xipone-\xippone}{\xipone+\xippone}+\Phi''(\eta)\right),\notag          \\
     \partial^2_\gamma
    F(x,\gamma)&=-\frac{1+\gamma^2}{(1-\gamma^2)^2}+\frac{h^2}{2\xippone}\left(1+2\frac{\xippone}{\xipone}+\Phi''(\eta)\right)
    \notag\\
    &= \frac{h^2}{2\xippone}\left(-\frac{2\xippone(\xipone+h^2)(\xipone+2h^2)}{h^2(\xipone)^2}+1+2\frac{\xippone}{\xipone}+\Phi''(\eta)\right),  \\
     \partial_x\partial_{\gamma} F(x,\gamma)&=\frac{h}{2\xippone}\left(1+\Phi''(\eta)\right). \notag
 %    \Phi''(\eta_*)&=-\frac{\eta_*}{\sqrt{\eta_*^2-2}}=-\frac{\xipone+h^2+\xippone}{\xipone+h^2-\xippone}\notag.
  \end{align}
   %\begin{align}
    % \partial^2_x
    %F(x,\gamma)&=\frac{1}{2\xippone}\left(1-2\frac{\xippone}{\xipone+\xippone}+\Phi''(\eta)\right),\\
    %&=\frac{1}{2\xippone}\left(1-2\frac{\xippone}{\xipone+\xippone}-\frac{(\xipone+\xippone+h^2)^2}{2\xippone (\xipone+h^2)}\right), \notag   %       \\
     %\partial^2_\gamma
    %F(x,\gamma)&=-\frac{1+\gamma^2%}{(1-\gamma^2)^2}+\frac{h^2}{2\xippone}\left(1+2\frac{\xippone}{\xipone}+\Phi''(\eta)\right)
%    \notag, \\
    %&=-\frac{(\xipone+h^2)(\xipone+2h^2)}{(\xipone)^2} %+\frac{h^2}{2\xippone}\left(1+2\frac{\xippone}{\xipone}-\frac{(\xipone+\xippone+h^2)^2}{2\xippone %(\xipone+h^2)}\right),\\
    % \partial_x\partial_{\gamma} F(x,\gamma)&=\frac{h}{2\xippone}\left(1+\Phi''(\eta)\right)\\
     %&=\frac{h}{2\xippone}\left(1-\frac{(\xipone+\xippone+h^2)^2}{2\xippone (\xipone+h^2)}\right) . \notag
%  \end{align}
  We now plug in $\gamma=\gamma^*$ and $x=x_*$, so that $\eta=\eta_*$ (recall the formulas from \eqref{eq:etastar}).
  Using $\Phi''(\eta)=-\frac{\eta}{\sqrt{\eta^2-2}}$ for $\eta>\sqrt{2}$ and $\sqrt{\eta_*^2-2}=\frac{\xipone+h^2-\xippone}{\sqrt{2\xippone(\xipone+h^2)}}$(recall \eqref{etastar_at_least_root2}), one verifies that  $\Phi''(\eta_*)=-\frac{\xipone+h^2+\xippone}{\xipone+h^2-\xippone}$. Using also $\gamma_*^2 = \frac{h^2}{\xipone + h^2}$ and simplifying one obtains
  \begin{align}
     \partial^2_x
    F(x_*,\gamma_*)&=-\frac{2\xipone+h^{2}}{\left(h^{2}+\xipone-\xippone\right)\left(\xipone+\xippone\right)},\notag          \\
     \partial^2_\gamma
    F(x_*,\gamma_*)&=-\frac{\left(\xipone+2h^{2}\right)\left(\xipone+h^{2}\right)}{\left(\xipone\right)^{2}}+h^{2}\frac{h^{2}-\xippone}{\xipone\left(h^{2}+\xipone-\xippone\right)}
    \notag, \\
     \partial_x\partial_{\gamma} F(x_*,\gamma_*)&=-\frac{h}{h^{2}+\xipone-\xippone}. \notag
 %    \Phi''(\eta_*)&=-\frac{\eta_*}{\sqrt{\eta_*^2-2}}=-\frac{\xipone+h^2+\xippone}{\xipone+h^2-\xippone}\notag.
  \end{align}
Using these expressions in the formula for the determinant of a $2\times 2$ matrix and extracting factors $h^2+\xipone-\xippone$, $(\xipone)^2$ and $\xipone+\xippone$ one obtains that $\operatorname{det} \nabla^2 F(x_*,\gamma_*)$ equals
\begin{equation}\label{eq: det exp}
\begin{split}
&\frac{1}{h^{2}+\xipone-\xippone} \left( \frac{2\xipone+h^{2}}{\xipone+\xippone}\left(\frac{\left(\xipone+2h^{2}\right)\left(\xipone+h^{2}\right)}{\left(\xipone\right)^{2}}-\frac{h^2(h^{2}-\xippone)}{\xipone\left(h^{2}+\xipone-\xippone\right)}\right)-\frac{h^{2}}{h^{2}+\xipone-\xippone} \right)\\
&=\frac{1}{\left(h^{2}+\xipone-\xippone\right)\left(\xipone\right)^{2}\left(\xipone+\xippone\right)}
\bigg(\left(2\xipone+h^{2}\right)\left(\xipone+2h^{2}\right)\left(\xipone+h^{2}\right)
\\&\quad-\frac{h^2\left(h^{2}-\xippone\right)\xipone}{h^{2}+\xipone-\xippone}\left(2\xipone+h^{2}\right)-\frac{h^{2}\left(\xipone\right)^{2}\left(\xipone+\xippone\right)}{h^{2}+\xipone-\xippone}\bigg).
\end{split}
\end{equation}
The last two terms equal
\begin{equation*}
h^{2}\frac{\left(h^{2}-\xippone\right)\xipone}{h^{2}+\xipone-\xippone}\left(2\xipone+h^{2}\right)+\frac{h^{2}\left(\xipone\right)^{2}\left(\xipone+\xippone\right)}{h^{2}+\xipone-\xippone}
=\frac{h^{2}\left(h^{2}-\xippone\right)\xipone\left(2\xipone+h^{2}\right)+h^{2}\left(\xipone\right)^{2}\left(\xipone+\xippone\right)}{h^{2}+\xipone-\xippone}.
\end{equation*}
Remarkably a factor of $h^{2}+\xipone-\xippone$ can be pulled out of the numerator by writing
\begin{equation}
\begin{array}{l}
h^{2}\left(h^{2}-\xippone\right)\xipone\left(2\xipone+h^{2}\right)+h^{2}\left(\xipone\right)^{2}\left(\xipone+\xippone\right)\\
=h^{2}\left(h^{2}+\xipone-\xippone\right)\xipone\left(2\xipone+h^{2}\right)-h^{2}\left(\xipone\right)^{2}\left(2\xipone+h^{2}\right)+h^{2}\left(\xipone\right)^{2}\left(\xipone+\xippone\right)\\
=h^{2}\left(h^{2}+\xipone-\xippone\right)\xipone\left(2\xipone+h^{2}\right)-h^{2}\left(\xipone\right)^{2}\left(\xipone+h^{2}-\xippone\right)\\
=(h^{2}+\xipone-\xippone)\left(h^{2}+\xipone\right)h^{2}\xipone.
\end{array}
\end{equation}
Thus the determinant \eqref{eq: det exp} becomes
\begin{equation}
\frac{\left(\xipone+h^{2}\right)\left\{ \left(2\xipone+h^{2}\right)\left(\xipone+2h^{2}\right)-h^{2}\xipone\right\} }{\left(h^{2}+\xipone-\xippone\right)\left(\xipone\right)^{2}\left(\xipone+\xippone\right)}.
\end{equation}
By multiplying out and completing the square the second term of the numerator simplifies to $2(\xipone+h^2)^2$ and we obtain
\begin{equation}\label{eq: det final form}
\operatorname{det} \nabla^2 F(x_*,\gamma_*) = \frac{2\left(\xipone+h^{2}\right)^{3}}{\left(h^{2}+\xipone-\xippone\right)\left(\xipone\right)^{2}\left(\xipone+\xippone\right)}.
\end{equation}
Using \eqref{eq:etastar}, $\eta_*=\frac{1}{\sqrt{2}}(\frac{1}{z}+z)$ and $\sqrt{\eta_*^2-2}=\frac{1}{\sqrt{2}}(\frac{1}{z}-z)$ with $z=\sqrt{\frac{\xippone}{\xipone+h^2}}\leq 1$ (recall \eqref{eta representation}), the remaining terms simplify to
\begin{align}
  &(1-\gamma_*^2)^{-{3}/{2}}=
  \xipone^{-3/2}\,(\xipone+h^2)^{3/2},\\
  &\frac{1}{(\eta_*^2-2)^{1/4}
    (|\eta_*|+\sqrt{\eta_*^2-2})^{1/2}}=\sqrt{\frac{\xippone}{\xipone+h^2-\xippone}}.
  \label{eq:endaa}
\end{align}
Plugging \eqref{eq: det final form}--\eqref{eq:endaa} into
\eqref{eq:chap5comp1} we obtain $\E[\mathcal N_N]\sim 2$, which completes
the proof of
\eqref{eq:expNtriv}.
\end{proof}

\section{Characterization of maximum}\label{sec:maxchar}

In the final section of this paper we prove the asypmtotic equalities
\eqref{eq:main_res_GS_energy}--\eqref{eq:eval_hessian} describing the
properties of the field at its maximizer. The proofs are based on
the following two lemmas.

\begin{lemma}
  \label{lem:pE}
  If there exists $p\geq 2$ such that $\xi(s)=a_p s^p$, then
  $p_{x,\gamma}(E)=\1_E\big(\frac{x\xipone}{\xipone+\xippone}+h\gamma\big)$.
  Otherwise, for $x\in\R$ and $E\subset \R$,
  \begin{align}
    p_{x,\gamma}(E)=\frac{1}{\sqrt{2\pi} J}\int_{E}
    \exp\bigg(-\frac{N}{2J^2}
        \bigg(y-\bigg(\frac{\xipone}{\xipone+\xippone}x
            +h\gamma\bigg)\bigg)^2\bigg) \dd y,\label{p__expression}
  \end{align}
  where $J^2=\frac{\xione(\xipone+\xippone)-\xipone^2}{\xipone +\xippone}$
  is a positive number.
\end{lemma}

\begin{proof}
  Using Lemma~\ref{lem:covariances_of_HN},
  \begin{equation}
    \E \dN H_N(\sigma)^2= N(\xipone+\xippone),\quad \E H_N(\sigma)^2=
    N\xione,
    \quad \E \partial_r H_N(\sigma) H_N(\sigma) =N \xipone.
  \end{equation}
  Using standard Gaussian conditioning formulas, this implies
  that, conditionally on $\dN H_N(\sigma)$,  $H_N(\sigma)$ is a Gaussian
  random variable with mean
  $\frac{\xipone}{\xipone+\xippone}\dN H_N(\sigma)$ and variance
  $N J^2$.
  By Jensen's inequality we have
  \begin{equation*}
  \begin{split}
      \xipone + \xippone &= \sum_{p\ge 1} p^2 a_p=\left(\sum_{p\ge 1} a_p\right) \frac{\sum_{p\ge 1} p^2 a_p}{\sum_{p\ge 1} a_p}\\
    &\ge  \left(\sum_{p\ge 1} a_p\right) \left(\frac{\sum_{p\ge 1} p
          a_p}{\sum_{p\ge 1} a_p }\right)^2
    = \frac{(\xipone)^2}{\xione},
  \end{split}
  \end{equation*}
  with equality only if $\xi(s)=a_p s^p$ for some $p\geq 2$. It follows
  that if $\xi$ takes this form, then $J=0$, and otherwise $J>0$. If $J=0$,
  then $H_N(\sigma)=\frac{\xipone}{\xipone+\xippone} \dN H_N(\sigma)$
  almost surely, and thus
$p_{x,\gamma}(E)=\1_E(\frac{x\xipone}{\xipone+\xippone}+h\gamma)$. If
$J^2>0$, we obtain \eqref{p__expression} as claimed.
\end{proof}

\begin{lemma}
  \label{lem:exponential_decay_complexity}
  Assume that $h^2> \xippone-\xipone$. Recall $x_*$ and
  $\gamma_*$ from \eqref{eq:etastar} and set
    $y_* =\sqrt{\xipone+h^2}$.
  Then for all closed sets $\Gamma\subset[-1,1]$ and $R,E\subset\R$ with
  $\pm(\gamma_*,x_*,y_*)\notin \Gamma\times R\times E$, we have
  $\lim_{N\to\infty}\E[\mathcal{N}_N(\Gamma, R, E)]= 0$.
\end{lemma}

\begin{proof}
We first assume that $\xi(s) $ is not of the form $a_p s^p$.
  By Proposition~\ref{pro:exactN}, Lemmas~\ref{lem:pE}
  and~\ref{lem:expass},
  \begin{equation}
    \label{INTEGRAND_ER_Gamma}
    \begin{split}
      &\E\big[|\mathcal N_N(\Gamma ,R,E) |\big]                                                        \\
      & =
      \exp{\left(N\left(\frac{1}{2}\ln{\left(\frac{\xippone}
                {\xipone}\right)}-\frac{h^2}{2\xipone}+o(1)\right)\right)}
      \int_R\int_\Gamma
      \frac{e^{NG(x,\gamma)}}{(1-\gamma^2)^{\frac{3}{2}}}
      \rho_N\left( \frac{x + h \gamma}{\sqrt{2
            \xippone}} \right) p_{x,\gamma}(E) \dd \gamma \dd x
      \\ & \leq\exp{\left  (N\left (
            \frac{1}{2}\ln{\left(\frac{\xippone} {\xipone}\right)}
            -\frac{h^2}{2\xipone}+o(1)\right)\right)}
      \\ & \quad\times \int_E \int_R \int_\Gamma
      \exp{\left(N\left(F(x,\gamma)-\frac{1}{2J^2}
            \left(y-\left(\frac{\xipone x}{\xipone+\xippone}
                +h\gamma\right)\right)^2+o(1)\right)\right)} \dd \gamma \dd x \dd
      y,
    \end{split}
  \end{equation}
  By noting that
  $y_*=\frac{\xipone}{\xipone+\xippone}x_*+h\gamma_* $ and using
  Proposition~\ref{pro:maxF}(i), if
  $\pm\left(\gamma_*,x_*,y_*\right)\notin \Gamma\times R\times E$, then the
  maximum of the exponent in the integrand of \eqref{INTEGRAND_ER_Gamma}
  over $R\times \Gamma\times E$ is strictly smaller than
  $F(x_*,\gamma_*) = -\left( \frac{1}{2}\ln{\left(\frac{\xippone}
                {\xipone}\right)}-\frac{h^2}{2\xipone}\right)$ (recall \eqref{eq:Fetastar}),
  since $R\times \Gamma\times E$ is a closed set. Using \eqref{ineq: tailest} we see that the tail of the integral plays no role, and thus $\E[\mathcal{N}_N(\Gamma, R,E)]\to 0$.
  The proof in the case $\xi (x) = a_p s^p$ is similar and simpler and is
  left to the reader.

   %and that the main contribution $\E[\mathcal{N}_N([-1,1], \mathbb{R}, \mathbb{R})]\to 2$ by Theorem \ref{thm:sec_thm} i)
\end{proof}

We can now prove claims
\eqref{eq:main_res_GS_energy}--\eqref{eq:eval_hessian} of
Theorem~\ref{thm:main_thm}. Recall that this theorem deals with the
trivial regime, that is we assume  $h^2>\xippone-\xipone$ for the
rest of this section.

\begin{proof}[Proof of \eqref{eq:main_res_GS_energy}]
  Taking $\varepsilon >0$ and
  applying Lemma~\ref{lem:exponential_decay_complexity} with
  $E=\{y\in \R : ||y|-y_*|\ge \e\}$, $R=\R$ and $\Gamma=[-1,1]$ we  obtain
  \begin{equation*}
    \E[|\{\sigma\in S_{N-1}:\nabla_\sph
        H^h_N(\sigma)=0,
        ||N^{-1}H_{N}^h(\sigma)|-y_*|\ge\e\}|]\to 0.
  \end{equation*}
  Moreover, for any $\sigma \in S_{N-1}$ with $\sigma \cdot \uN = 0$, one
  has $\P(|N^{-1}H_{N}^h(\sigma)|>\e) \to 0$ as $N\to\infty$. As
  $\sigma^*$ is the global maximum of $H^h_N$, we thus have
  $N^{-1}H_N^h(\sigma^*)>-\varepsilon $ with probability tending to one. This implies
  \begin{equation}\label{radical convergence}
    \lim_{N\to\infty}\P(|N^{-1}H_{N}^h(\sigma^*)-y_*|>\e)= 0,
  \end{equation}
  which proves \eqref{eq:main_res_GS_energy}.
\end{proof}

\begin{proof}[Proof of \eqref{eq:main_res_overlap_with_ext_field}]
  By a similar argument as in the last proof, taking
  $\Gamma=[\gamma_*-\e,\gamma_*+\e]^c$ and $E=[y_*-\e,y_*+\e]$, so that
  $\Gamma \times \mathbb R \times E$ does not contain $\pm (\gamma_*,x_*,y_*)$,
  \begin{align*}
    \P&(|{\uN}\cdot \sigma^*-\gamma_*|>\e)\\
    &\le   \P(|{\uN}\cdot \sigma^*-\gamma_*|>\e,\,|N^{-1}H_{N}^h(\sigma^*)-y_*|\le \e)+\P(|N^{-1}H_{N}^h(\sigma^*)-y_*|>\e)\\
    & \le \E \left[ \mathcal{N}_N( \Gamma, \R ,E) \right]
    +\P(|N^{-1}H_{N}^h(\sigma^*)-y_*|>\e) \to 0,
  \end{align*}
  where in the last step we have used \eqref{radical convergence} and
  Lemma~\ref{lem:exponential_decay_complexity}.
  This proves \eqref{eq:main_res_overlap_with_ext_field}.
\end{proof}

\begin{proof}[Proof of \eqref{eq:rad_deriv}]
  Repeating the same argument, for
  $R=[x_*-\e,x_*+\e]^c$, $E=[y_*-\e,y_*+\e]$
  \begin{equation}
    \begin{split}
      \P&(|N^{-1}\dN H_{N}(\sigma^*)-x_*|>\e) \\
      &\le   \P(|N^{-1}\dN H_{N}(\sigma^*)-x_*|>\e,\,|N^{-1}H_{N}^h(\sigma^*)-y_*|\le \e)+\P(|N^{-1}H_{N}^h(\sigma^*)-y_*|>\e)\\
      &\le \E \left[ \mathcal{N}_N( [-1,1], R ,E) \right]+\P(|N^{-1}H_{N}^h(\sigma^*)-y_*|>\e)\to 0.
    \end{split}
  \end{equation}
  This implies that $N^{-1}\dN H_N(\sigma^*) \to x_*$ in probability.
  Recalling that
  $\dN H_N^h(\sigma)= \dN H_N(\sigma)+ N\,h \uN\cdot \sigma $ and using
  $ x_*+h\gamma_*=\frac{\xipone+\xippone+h^2}{\sqrt{\xipone+h^2}}$, we
  obtain \eqref{eq:rad_deriv}.
\end{proof}

To prove \eqref{eq:eval_hessian} we need a standard large deviation estimate for the largest eigenvalue of a GOE random matrix. For a matrix $A$ let $\lambda_\Max(A)$ denote the largest eigenvalue. Then
\begin{equation}
    \label{eq:lambdamaxLDGOE}
    \begin{array}{c}
    \text{for all }\e>0\text{ there is a }\delta>0 \text{ such that for $N$ large enough}\\
    \P(|\lambda_\Max(\GOE_{N}(N^{-1}))-\sqrt{2}|>\e)\leq
    e^{-\delta N},
    \end{array}
  \end{equation}
see e.g.~\cite[(2.6.31)]{anderson2010introduction}. We also need the following lemma.
\begin{lemma}
\label{lem:determinant_square} For all $\delta>0$ it holds for $N$ large enough and $\sqrt{2}+\delta \le x \le \delta^{-1}$ that
\begin{equation}\label{eq:determinant_square}
\mathbb{E}\left[\left|\det\left(x \Id_{N}+\GOE_{N}\left(N^{-1}\right)\right)\right|^2 \right]\le e^{\delta N}\mathbb{E}\left[\left|\det\left(x \Id_{N}+\GOE_{N}\left(N^{-1}\right)\right)\right|\right]^2.
\end{equation}
\end{lemma}
\begin{proof}%[Proof of Lemma~\ref{lem:determinant_square}]
Note that
$$\det\left(x \Id_{N}+\GOE_{N}\left(N^{-1}\right)\right) = \exp\left( N\int\ln{|\eta-\lambda|} \dd L_N (\lambda) \right),$$
where $L_{N}=\frac{1}{N}\sum_{i=1}^{N}\delta_{\lambda_i}$  is the empirical measure of the eigenvalues $(\lambda_i)_{i=1}^{N}$ of $\GOE_N(N^{-1})$. We follow \cite[Lemma~16]{subag2017complexity} in approximating $\ln$ by a bounded continuous function, and applying the the large deviation principle for the empirical spectral measure with speed $N^2$. For $\kappa>1$, define the function
  \begin{align*}
    \ln_{\kappa} x=
    \begin{cases}
      -\ln{\kappa},            & \text{if }  x<\kappa^{-1},            \\
      \ln{ x},                 & \text{if }  \kappa^{-1}\leq x<\kappa, \\
      \ln{\kappa},\qquad       & \text{if }  x\geq \kappa.
    \end{cases}
  \end{align*}
  Note that $\ln{x}\leq \ln_\kappa{x}$ for $x\leq \kappa$. Set
  $|\lambda|_\Max=\max_{1\leq i\leq N}|\lambda_i|$. For $x \le \delta^{-1}$ we have
  \begin{equation}
    \begin{split}
      \label{eq:iiiiiiiiii}
      &\E\left[\left|\det\left(x\,\Id_{N} +\GOE_{N}(N^{-1})\right)\right|^2\right]\\
      &=\E\left[ \exp{\left(2N\int\ln{|x-\lambda|} \dd L_{N}(\lambda)
      \right)}\right]\\
     &\le \E\left[ \exp{\left(2N \int\ln_{\kappa}{|x-\lambda|} \dd L_{N}(\lambda)
            \right)}+\exp{\left(2 N \ln{(|x|+|\lambda|_\Max)}
      \right)}\1_{\{|\lambda|_\Max+|x|>\kappa\}}\right]\\
      &\leq 2 \E\left[ \exp{\left(2N\int\ln_{\kappa}{|x-\lambda|} \dd L_{N}(\lambda)\right)}\right],
    \end{split}
  \end{equation}
where the last inequality follows by taking $\kappa$ large enough and using the estimate $\P( |\lambda|_{\max} \ge M )\le e^{-N {M^2}/{9}}$ (see Lemma~6.3 in \cite{arous2001aging}).

  We now apply the large deviation principle (with speed $N^2$) for the empirical spectral measure, see e.g.~\cite[Theorem~2.6.1]{anderson2010introduction}.
  Consider the set
    \begin{equation*}
      F=\left\{\mu\in M_1(\R):~ \left| \int \ln_\kappa|x-\lambda| \dd \mu(\lambda)-
        \int_{-\sqrt 2}^{\sqrt 2} (2\pi)^{-1}\ln_\kappa|x-\lambda| {\sqrt{2-\lambda^2}} \dd\lambda \right| > \frac{\delta}{8}\right\},
    \end{equation*}
    where $M_1(\R)$ stands for set of probability measures on $\R$. Since $\ln_\kappa(\cdot)$ is a bounded continuous function the large deviations principle implies that $\P( L_{N} \notin F) \le e^{-c'N^2}$ for some $c'>0$. Therefore
   the first
  expectation on the right-hand side of \eqref{eq:iiiiiiiiii} can be bounded from above
  \begin{equation}
    \label{first_term_estimate_in_square_determinante}
    \begin{split}
      \E&\left[ \exp{\left(2N\int\ln_{\kappa}{|x-\lambda|}
            \dd L_{N}(\lambda)\right)}\right]\\
      &\leq
      e^{N\delta/4}\exp{\left(2N\int_{-\sqrt{2}}^{\sqrt{2}} (2\pi)^{-1}\ln_{\kappa}{|x-\lambda|}
          {\sqrt{2-\lambda^2}} \dd\lambda \right)}+ e^{- c' N^2}\\
      &\leq 2 e^{N\delta/4}\left( \exp{\left(N\int_{-\sqrt{2}}^{\sqrt{2}}
          (2\pi )^{-1}
          \ln{|x-\lambda|}
          {\sqrt{2-\lambda^2}}
          \dd\lambda \right)}\right)^2.
    \end{split}
  \end{equation}
  The claim then follows since for $x > \sqrt{2} + \delta$ and $\kappa$ large enough
    \begin{equation}
    \begin{split}
      &\exp{\left(N\int_{-\sqrt{2}}^{\sqrt{2}}
          (2\pi )^{-1}
          \ln_\kappa {|x-\lambda|}
          {\sqrt{2-\lambda^2}}
          \dd\lambda \right)}\\
       &\le 2 \E\left[ \exp{\left(N\int_{-\sqrt{2}}^{\sqrt{2}}
          (2\pi )^{-1}
          \ln_\kappa{|x-\lambda|}
          {\sqrt{2-\lambda^2}}
          \dd\lambda \right)}1_{L_N \in F, |\lambda|_{\max} \le \sqrt{2} + \delta/2} \right]\\
        &\le 2 e^{N\delta/8} \E\left[ \exp{\left(N\int
              \ln_\kappa{|x-\lambda|}
              \dd L_N(\lambda) \right)}1_{L_N \in F, |\lambda|_{\max} \le \sqrt{2} + \delta/2} \right]\\
        &\le 2 e^{N\delta/8} \E\left[ \exp{\left(N\int
              \ln{|x-\lambda|}
              \dd L_N(\lambda) \right)} \right]\\
       & =2 e^{N\delta/8} \E\left[ \left|\det\left(x \Id_{N}+\GOE_{N}\left(N^{-1}\right)\right)\right| \right],
    \end{split}
  \end{equation}
  where we used the fact that $\ln_\kappa z \le \ln z$ for $z \ge \kappa^{-1}$ as well as $\P\left( L_N \notin F\right)\to 0$ and $\P\left( |\lambda|_{\max} > \sqrt{2} + \delta/2 \right)\to0$ (see \eqref{eq:lambdamaxLDGOE}).
\end{proof}

\begin{proof}[Proof of \eqref{eq:eval_hessian}]
  Define
  \begin{equation*}
    {\mathbf M}_{N-1}(\sigma):=\frac{1}{N\sqrt{2\xippone}}
    \nabla^2 H_N^h(\sigma)|_\sph,
  \end{equation*}
  so that $\mathbf M_{N-1}(\sigma ) \overset d = \GOE_{N-1}(N^{-1})$, by
  Lemma~\ref{lem:law_of_grads}(c).  Using \eqref{eq:covHess},
  \begin{align*}
    \nabla^2_\sph H_{N}^h(\sigma)&= -\dN
    H_{N}^h(\sigma)\,\Id_{N-1}+ \nabla^2
    H_N^h(\sigma)|_\sph\\
    &=-\dN H_{N}^h(\sigma)\,\Id_{N-1}+
    N\sqrt{2\xippone}{\mathbf M}_{N-1}(\sigma).
  \end{align*}
  Recalling \eqref{eq:rad_deriv},  to prove \eqref{eq:eval_hessian}
  it hence suffices to show that
  \begin{equation}
    \label{eq:lambdasigmastar}
    \lim_{N\to\infty}\lambda_\Max({\mathbf M}_{N-1}(\sigma^*)) =
    \sqrt{2},\quad
    \text{in probability}.
  \end{equation}

  To show \eqref{eq:lambdasigmastar}, we recall $\eta_*$ from
  \eqref{eq:etastar}, and define
  \begin{equation*}
    \kE_{\e}=\left\{\sigma \in S_{N-1} : |\lambda_\Max({\mathbf M}_{N-1}(\sigma))-\sqrt{2}|
      >\e,\,
      \left|\frac{\dN H_{N}^h(\sigma)}{\sqrt{2\xippone}\, N}
      -\eta_*\right|<\e\right\}.
  \end{equation*}
  By the Kac-Rice formula as in the proof of
  Proposition~\ref{pro:exactN}
  \begin{equation}
  \begin{split}
    \E&\left[\left|\left\{ \sigma\in \mathcal E_\varepsilon :
        \nabla_\sph
        H_{N}^h(\sigma)=0  \right\}\right|\right]\\
    &= \int_{S_{N-1}} \,f_{\nabla_\sph
      H^h_N(\sigma)}(0)\, \E\left[\left|{\rm
        det}\left(-\dN H_{N}^h(\sigma)\, \Id_{N-1}
        +N\sqrt{2\xippone}{\mathbf M}_{N-1}(\sigma)\right)\right|
      \1_{\mathcal E_\varepsilon }(\sigma)\right]\dd \sigma. \label{eqn:pppppppp}
      \end{split}
  \end{equation}
  To compute the expectation  inside the integral, recall that $\dN H_{N}^h(\sigma)$
  and ${\mathbf M}_{N-1}(\sigma)$ are independent by
  Lemma~\ref{lem:law_of_grads}(a). Hence,
  \begin{equation}
    \label{lambda_max_computation}
    \begin{split}
      \E&\left[\left|\det\left(-\dN
          H_{N}^h(\sigma)\, \Id_{N-1} + N\sqrt{2\xippone}
          {\mathbf M}_{N-1}(\sigma)\right)\right| \1_{\mathcal
          E_\varepsilon }(\sigma )\right]        \\
      &= (2\xippone N^2)^{\frac{N-1}{2}}  \E\left[\left|\det\left(-\frac{\dN
            H_{N}^h(\sigma)\, }{N\sqrt{2\xippone}}\Id_{N-1} +
          {\mathbf M}_{N-1}(\sigma)\right)\right| \1_{\kE_\e}(\sigma )\right]        \\
      & =(2\xippone N^2)^{\frac{N-1}{2}} \int_{\eta_*-\e}^{\eta_*+\e} \E\left[\left|\det\left(-\eta\,
          \Id_{N-1}+{\mathbf M}_{N-1}(\sigma)\right)\right|
        \1\{|\lambda_\Max({\mathbf M}_{N-1}(\sigma))-\sqrt{2}|>\e\}\right]
      \\&\qquad\times
      f_{\frac{\dN
          H_{N}^h(\sigma)}{N\sqrt{2\xippone}}}(\eta) \dd \eta      \\
      & \leq (2\xippone N^2)^{\frac{N-1}{2}} \left(\int_{\eta_*-\e}^{\eta_*+\e} \E\left[\left|\det\left(-\eta
            \Id_{N-1} +  {\mathbf M}_{N-1}(\sigma)
            \right)\right|^2\right]^{1/2} f_{\frac{\dN
            H_{N}^h(\sigma)}{N\sqrt{2\xippone}}}(\eta) \dd \eta \right)                \\
      & \qquad \times \P(|\lambda_\Max({\mathbf
            M}_{N-1}(\sigma))-\sqrt{2}|>\e)^{1/2},
    \end{split}
  \end{equation}
  where in the last step we used the Cauchy-Schwarz inequality.
  
  Using \eqref{eq:lambdamaxLDGOE} we have for some $\delta$ that
  \begin{equation}
    \label{eq:lambdamaxLD}
    \P(|\lambda_\Max({\mathbf M}_{N-1}(\sigma))-\sqrt{2}|>\e)\leq
    e^{-\delta N},
  \end{equation}
  (using $N-1$ in place of $N$ and multiplying both sides in the event of \eqref{eq:lambdamaxLDGOE} by a factor to deal with the small mismatch between matrix dimension $N-1$ and variance $N^{-1}$ of entries of ${\mathbf M}_{N-1}(\sigma)$ in \eqref{eq:lambdamaxLD}).
  Furthermore
      \begin{equation}\label{eq: det_square_estimate}
      \E\left[\left|\det\left(-\eta\,\Id_{N-1} +{\mathbf
            M}_{N-1}(\sigma)\right)\right|^2\right]\leq e^{N \delta/2 } \E\left[\left|
        \det\left(-\eta\,\Id_{N-1} +{\mathbf
            M}_{N-1}(\sigma)\right)\right|\right]^2,
    \end{equation}
for large enough $N$ by Lemma \ref{lem:determinant_square} since if $h^2>\xippone-\xipone$ then $\eta_*>\sqrt{2}$ by \eqref{etastar_at_least_root2}, dealing similarily with the mismatch of matrix dimension and variance in \eqref{eq: det_square_estimate}.
  Using \eqref{eq:lambdamaxLD} and \eqref{eq: det_square_estimate}, for all $N$ large
  enough,
  the right-hand side of \eqref{lambda_max_computation} is bounded by
  \begin{align*}
    e^{-N\delta/4  }& (2\xippone N^2)^{\frac{N-1}{2}}\int_{\eta_*-\e}^{\eta_*+\e} \E\left[\left|
      \det\left(-\eta\,\Id_{N-1} + {\mathbf
          M}_{N-1}(\sigma)\right)\right|\right]f_{\frac{\dN
        H_{N}^h(\sigma)}{N\sqrt{2\xippone}}}(\eta) \dd \eta \\
    &\leq e^{- N\delta /4} \E\left[\left|\det\left(- \dN
        H_{N}^h(\sigma) \Id_{N-1}+ N \sqrt{2\xippone} {\mathbf
          M}_{N-1}(\sigma)\right)\right|\right].
  \end{align*}
  Plugging this into \eqref{eqn:pppppppp}, we obtain for $N$ large enough
  \begin{align*}
    \P\left(\sigma^* \in E_\varepsilon \right)&\le
    \E\left[\left|\left\{ \sigma\in \mathcal E_\varepsilon :
        \nabla_\sph
        H_{N}^h(\sigma)=0  \right\}\right|\right]\\
    &\leq e^{-  N\delta /4}	\int \E\left[\left|\det\left(- \dN
        H_{N}^h(\sigma) + N \sqrt{2\xippone} {\mathbf
          M}_{N-1}(\sigma)\right)\right|\right] f_{\nabla_\sph
      H_{N}^h(\sigma)}(0)\dd \sigma\\
    &= e^{-\delta /4 N}\E[\mathcal{N}_N] \xrightarrow{N\to\infty}0,
  \end{align*}
  since $\lim_{N\to\infty}\E[\mathcal{N}_N]= 2$.
  Therefore for any $\e>0$
  \begin{equation*}
    \P(|\lambda_\Max({\mathbf M}_{N-1}(\sigma^*))-\sqrt{2}|>\e)
    \leq \P\left(\sigma^*\in \mathcal E_\varepsilon  \right)+\P\left(\left|\frac{\dN H_{N}^h(\sigma^*)}{\sqrt{2\xippone}}- \eta_*\right|\geq \e\right)
    \xrightarrow{N\to\infty}0,
  \end{equation*}
  by the previous display and \eqref{eq:rad_deriv}. This proves
  \eqref{eq:lambdasigmastar} and thus
  \eqref{eq:eval_hessian}.
\end{proof}

{{\bf Acknowledgements}: We thank Valentina Ros for useful
    discussions about the work \cite{ros2019complex}, and Antti Knowles
    and Gaultier Lambert for their helpful advice regarding the random
    matrix estimates used in this article.}
\appendix

\section{Random matrix estimates}.

In the first part of the appendix, we prove several auxiliary results about GOE random matrices
that were used in the main part of the paper.

\begin{proof}[Proof of Lemma~\ref{lem:detGOE}]
  The proof builds on the argument in  Lemma~3.3 in
  \cite{auffinger2013random}. Recall that
  $\lambda^{N,a}_1\geq \cdots\geq \lambda_{N}^{N,a}$ denote the ordered
  eigenvalues of $\GOE_{N}(a)$. The distribution $Q_{n,a}$
  of $(\lambda_i^{N,a})_{i\leq N}$ can be written expicitly,
  see \cite[Theorem 3.3.1]{mehta2004random}:
  \begin{align}\label{explicit formula ev}
    Q_{N,a}(\dd \lambda)=\frac{N!}{Z_{N}(a)}e^{-\frac{1}{2a}\sum_{i=1}^N
      \lambda_i^2}\,\Delta_{N}(\lambda)\, \mathbf{1}\{\lambda_1< \cdots <
      \lambda_N\} \,\prod_{i=1}^N \dd \lambda_i,
  \end{align}
  where (see \cite[(3.3.10)]{mehta2004random})
  \begin{equation}
    Z_{N}(a)=(2\pi)^{N/2} a^{N(N+1)/4} \prod_{j=1}^N
    \frac{\Gamma\left(1+\frac{j}{2}\right)}{\Gamma\left(\frac{3}{2}\right)},
  \end{equation}
  and $\Delta_N(\lambda)=\prod_{1\leq i<j\leq N} |\lambda_i-\lambda_j|$
  is  the van der Monde determinant.  We write
  $Z_N=Z_N(N^{-1})$ and $Z_{N-1}'=Z_{N-1}(N^{-1})$ and define
  ${\mathbf{T}_{N}}=\{(x_i)_{i=1}^N\subset \R^N:~x_1<\cdots<x_N\}$.
  Then,
  \begin{equation}
    \label{middle_expression_of_determinant_GOE}
    \begin{split}
      \E[\det&(x\Id_{N-1}+\GOE_{N-1}(N^{-1}))]
       =\int \prod_{i=1}^{N-1} |x-\lambda_i |\, Q_{N-1,N^{-1}}(\mathrm d
        \lambda)                                                            \\
      & =\frac{(N-1)!}{Z_{N-1}'}\int \prod_{i=1}^{N-1} \left|x-\lambda_i
      \right|e^{-\frac{N}{2}\sum_{i=1}^{N-1}\lambda_i^2}\,
      \Delta_{N-1}(\lambda) \mathbf{1}_{\mathbf{T}_{N-1}}(\lambda)
      \,\prod_{i=1}^{N-1} \dd \lambda_i                                   \\
      & =\sum_{j=1}^N \frac{(N-1)!}{Z_{N-1}'}\int \prod_{i=1}^{N-1}
      \left|x-\lambda_i
      \right|\,e^{-\frac{N}{2}\sum_{i=1}^{N-1}\lambda_i^2}\,
      \Delta_{N-1}(\lambda)                                                       \\
      & \qquad\qquad\qquad\qquad \times \mathbf{1}_{\{\lambda_1<\cdots
        < \lambda_{j-1}< x < \lambda_j < \cdots < \lambda_{N-1}\}}
      \,\prod_{i=1}^{N-1} \dd \lambda_i,
    \end{split}
  \end{equation}
  with the convention that $\lambda_0=-\infty$ and
  $\lambda_N=\infty$.
  We note that
  $\Delta_{N-1}(\lambda)\,\prod_{i=1}^{N-1}
    |x-\lambda_i|=\Delta_{N}(\nu)$, with
  $\nu=\nu(\lambda,x)=(\lambda_1,\dots, \lambda_{j-1},x,\lambda_{j}, \dots, \lambda_{N})$.
  Having this in mind, since the dirac delta function $\delta(x-y)$ enable us to
  exchange $x$ and $y$ freely,
  \eqref{middle_expression_of_determinant_GOE} is equal to
  \begin{align}
    & \qquad \sum_{j=1}^N \frac{(N-1)!}{Z_{N-1}'}\int
    \delta(x-\nu_{j}) \,
    \exp{\left(-\frac{N}{2}\sum_{i\in\{1,\cdots,N\}\backslash\{j\}}\nu_i^
        2\right)}\,\Delta_{N}(\nu)\mathbf{1}_{\mathbf{T}_{N}}(\nu) \,\prod_{i=1}^{N}
    \dd \nu_i\notag       \\
    & =\frac{Z_N}{N\,Z_{N-1}'} e^{\frac{N}{2}x^2} \sum_{j=1}^N  \frac{N!}{Z_N}
    \int \delta(x-\nu_{j}) \,
    \exp{\left(-\frac{N}{2}\sum_{i=1}^N
        (\nu_i)^2\right)}\,\Delta_N(\nu)\,
   \mathbf{1}_{\mathbf{T}_{N}}(\nu) \,\prod_{i=1}^{N} \dd
    \nu_i\notag   \\
    & =\frac{Z_{N}}{Z_{N-1}'} e^{\frac{N}{2}x^2}   \int
    \left[\frac{1}{N}\sum_{j=1}^N \delta(x-\nu_{j})\right] Q_{N,N^{-1}}(\dd
      \nu)\notag\\
    & =\frac{Z_{N}}{Z_{N-1}'} e^{\frac{N}{2}x^2}   \int
        \left[\frac{1}{N}\sum_{j=1}^N \delta(x-\lambda_{j})\right] Q_{N,N^{-1}}(\dd
        \lambda),
  \label{asymptotic-formula}
  \end{align}
  where in the last line now $\lambda \in \mathbb{R}^N$. Since
  $\int_A\left(\sum_{i=1}^N\delta\left(x-\lambda_i\right)\right)
  \dd x= |\{1\leq i\leq N|~\lambda_i\in A\}|$, we have
  \begin{align*}
    \mu_{N,N^{-1}}(A)&=
    \E\left[\int_A \frac{1}{N}\sum_{i=1}^N\delta(x-\lambda_i^{N,N^{-1}}) \dd
    x \right]\\
    &= \int \int_A \left[\frac{1}{N}\sum_{j=1}^N \delta(x-\lambda_{j})\right]\dd x\,
    Q_{N,N^{-1}}(\dd \lambda) \\
    &=\int_A \int \left[\frac{1}{N}\sum_{j=1}^N \delta(x-\lambda_{j})\right]
    Q_{N,N^{-1}}(\dd \lambda) \dd x,
    \end{align*}
which implies
  \aln{\label{rhoN formula}
  \rho_N(x)=\rho_{N,N^{-1}}(x)=\int \left[\frac{1}{N}\sum_{j=1}^N
      \delta(x-\lambda_{j})\right] Q_{N,N^{-1}}(\dd \lambda).
      }
  Thus, the right-hand side of \eqref{asymptotic-formula} is equal to
  $\frac{Z_N}{Z'_{N-1}}e^{\frac{N}{2}x^2}\rho_N(x).$

  Note that
  \begin{align*}
    \frac{Z_{N}}{Z_{N-1}'}&= \sqrt{2\pi} N^{-N/2}
    \frac{\Gamma(1+\frac{N}{2})}{\Gamma\left(\frac{3}{2}\right)}\\
    &=\sqrt{2}\, N^{-(N-2)/2}\, \Gamma\left(\frac{N}{2}\right),
  \end{align*}
where we have used  $\Gamma(3/2)=\sqrt{\pi}/2$ and $\Gamma(1+x)=x\Gamma(x)$, which completes the proof.
\end{proof}

%\begin{lemma}[{\cite[(5.5.4)]{adler2009random}}]\label{lem:covariance
%    computation}
%  Let $(H(\sigma))_{\sigma\in \R^N}$ be a Gaussian process with  covariance
%  $C(\sigma,\tau)=\E\left[H(\sigma)H(\tau)\right]$ for $\sigma, \tau\in \R^N$.
%  Suppose that $C(\sigma,\tau)$ is smooth. Then, for any $m_1,m_2\in \N$, $1\leq
%  i_1,\cdots,i_{m_1}\leq N$ and $1\leq j_1,\cdots,j_{m_1}\leq N$,
%  $\frac{\partial^{m_1} H_N(\sigma)}{\partial \sigma_{i_1},\cdots \partial
%    \sigma_{i_{m_1}}} $ is still a Gaussian random variable and
%  \begin{align*}
%    \E\left[\frac{\partial^{m_1} H_N(\sigma)}{\partial \sigma_{i_1},\cdots \partial
%        \sigma_{i_{m_1}}} \frac{\partial^{m_2} H_N(\tau)}{\partial \tau_{i_1},\cdots
%        \partial \tau_{i_{m_2}}}\right]=\frac{\partial^{m_1+m_2}
%      C(\sigma,\tau)}{\partial \sigma_{i_1},\cdots \partial \sigma_{i_{m_1}}\cdots
%      \partial \tau_{i_1},\cdots \partial \tau_{i_{m_2}}}
%  \end{align*}
%\end{lemma}

%\SN{
To prove Lemma~\ref{lem:rhoapprox} we use a formula for $\rho_N$ in terms of Hermite polynomials. Let $\phi_n(x)=(2^n n! \sqrt{\pi})^{-\frac{1}{2}}\mathcal{H}_n(x) e^{-\frac{x^2}{2}}$, where $(\mathcal{H}_n(x))_{n\geq 0}$ are the Hermite polynomials.

\begin{lemma}
  %An explicit formula for $\rho_N(x)$ is given by  \cite[(7.2.32) and pp 511]{mehta2004random}, which in our normalization reads:
  It holds that
  \begin{equation}\label{eq: rho A B}
     \rho_N(x)=N^{-1/2}(A_N(x)+B_N(x))
  \end{equation}
    where
\begin{equation}\label{eq: A def}
    A_N(x)=\sum_{i=0}^{N-1} \phi_i(\sqrt{N}x)^2=N\phi_N(\sqrt{N}x)^2-\sqrt{N(N+1)}\phi_{N-1}(\sqrt{N}x)\phi_{N+1}(\sqrt{N}x),
\end{equation}
and
\begin{equation}\label{eq: B def}
    B_N(x)=S_N(x)+\alpha_N(x),\\
\end{equation}
for
\begin{equation}\label{eq: SN def}
    S_N(x)=\sqrt{\frac{N}{2}}\phi_{N-1}(\sqrt{N}x)J_N(x),\\
\end{equation}
and
\begin{equation}\label{eq: alpha def}
    \alpha_N(x)=
    \begin{cases}
      \phi_{N-1}(\sqrt{N}x)\left(\int^\infty_{-\infty} \phi_{N-1}(t)\dd t\right)^{-1} &\text{if $N$ is odd},\\
      0&\text{if $N$ is even}
    \end{cases}
\end{equation}
where
\begin{equation}
  \label{eq: JN def}
    J_N(x)=\int^\infty_{-\infty} \frac{{\rm sgn}(\sqrt{N}x-t)}{2}\phi_N(t)\dd t
    = \begin{cases}
    -  {\rm sgn}(x) \sqrt{N}\int^\infty_{|x|} \phi_N(\sqrt{N}\,t)\dd t &\text{if $N$ is odd,}\\
    {\rm sgn}(x) \sqrt{N}\int_0^{|x|} \phi_N(\sqrt{N}\,t)\dd t &\text{if $N$ is even}.
    \end{cases}
  \end{equation}
\end{lemma}
\begin{proof}
  This follows from \cite[(7.2.19), (7.2.27), (7.2.28), (7.2.30), (7.2.32) and pp 511]{mehta2004random} (after translating to our normalization). The second equality of \eqref{eq: A def} can be found on \cite[page 511]{mehta2004random}, and the last identity is due to $\phi_N$ and $\mathcal{H}_N$ being even functions for $N$ even and odd functions for $N$ odd.
\end{proof}

The proof of Lemma~\ref{lem:rhoapprox} is then based on applying the following bounds for Hermite polynomials.

\begin{lemma}
Fix $\delta>0$.
\begin{enumerate}
\item Uniformly for $x\in [0,\sqrt{2}(1-\delta)]$ we have
\begin{equation}\label{sc asymptotics}
  N^{-1/2} A_N(x) \to \frac{1}{2\pi}\sqrt{2-x^2}.
\end{equation}

\item Uniformly for $x\in [0,\sqrt{2}(1-\delta)]$ we have
\aln{\label{inside asymptotics}
  \phi_N(\sqrt{N}x)=O(N^{-1/4})
  }
\item Uniformly for $x\in [\sqrt{2}(1+\delta), \infty)$, we have
\aln{\label{outside asymptotics}
  \phi_N(\sqrt{N}x)=\frac{e^{N\Phi(x)}g(x)}{\sqrt{4\pi \sqrt{2N}}}\left(1+o(1)\right)\text{ where } g(x)=\left|\frac{x-\sqrt{2}}{x+\sqrt{2}}\right|^{1/4}+\left|\frac{x+\sqrt{2}}{x-\sqrt{2}}\right|^{1/4}.
  }
    \item  Uniformly for $x\in[\sqrt{2}(1-\delta),\sqrt{2}(1+\delta)]$,
\aln{\label{middle asymptotics}
  \phi_N(\sqrt{N}x)&=(2N)^{-1/4}\left( (x+\sqrt{2})(\sqrt{2}N^{2/3})^{1/4} |\hat{f}_N(x)|^{1/4}{\rm Ai}(f_N(x))(1+o(1)))\right.\nonumber\\
  &\qquad\qquad\left.- (x+\sqrt{2})^{-1}(\sqrt{2}N^{2/3})^{-1/4}  |\hat{f}_N(x)|^{-1/4} {\rm Ai}'(f_N(x))(1+o(1)) \right),
}
where ${\rm Ai}(x)$ is the Airy function, and $f_N(x)=\sqrt{2} N^{2/3} (x-\sqrt{2}) \hat{f}_N(x)$  with an analytic function $\hat{f}_N$ such that, if $\delta$ is small enough then there are constants $c < C$ such that $0 < c \le \hat{f}_N(x) < C < \infty$ uniformly in  $x\in [\sqrt{2}(1-\delta),\sqrt{2}(1+\delta)]$.
\item It holds that
\aln{\label{integral hermite fcn}
\int_{0}^\infty \phi_N(x) \dd x\sim (2N)^{-1/4},\,\int_{\R} \phi_{N-1}(x) \dd x =
\begin{cases}
  2(2N)^{-1/4}(1+o(1))& \text{if $N$ is odd}\\
  0& \text{if $N$ is even.}
\end{cases}
  }
\item Uniformly for $x\in\mathbb{R}$
\begin{equation}
\int_{-\infty}^x \phi_N(x) \dd x = O(N^{-1/4}),
\end{equation}
and
\ben{\label{estimate JN}
  \sup_{x\geq 0}|J_N(x)|=O(N^{-1/4}).
  }
\item There exists $c>0$ such that for $N$ large enough,
\ben{\label{(4.16)}
\int_{\sqrt{2}(1+\delta)}^\infty|\phi_N(\sqrt{N} x)|\dd x\leq e^{-c N}.
}
\end{enumerate}
\end{lemma}
\begin{proof}$ $

\begin{enumerate}
    \item The density of the expected empirical spectral distribution for the GUE ensemble (i.e. the object corresponding to $\rho_N$ for this ensemble) is precisely $N^{-1}A_N(x)$ \cite[(6.2.10)]{mehta2004random}. It is well-known that this density (whether for GUE or GOE) converges to the semi-circle law density $\frac{1}{2\pi}\sqrt{2-x^2}$ point-wise, and \cite[Theorem, page 16]{lindseyasymptotics} shows for the GUE that this convergence is uniform on compact subsets of $(-\sqrt{2},\sqrt{2})$.
\end{enumerate}

 The remaining estimates are from \cite{deift2007universality}, where they are given for general orthogonal polynomials in terms of the quantities $c_N, d_N,h_N(x)$ \cite[(2.3),(2.4),(2.6)]{deift2007universality}. Results for the standard Hermite polynomials are obtained by setting (in  the notation of \cite{deift2007universality}) $m=1$, $\kappa_2=1$, $\kappa_k=0$ for $k\neq 2$. In this special case  $c_N=\sqrt{2N}$, $d_N=0$, $h_N=4$ by \cite[pp 1501, Remark 3.]{deift1999strong}. To obtain the estimates for our normalization of the GOE the variable $x$ in the formulas of \cite{deift2007universality} should furthermore be replaced by $x/\sqrt{2}$.

\begin{enumerate}
    \setcounter{enumi}{1}
 \item
 This follows directly from \cite[pp 38 penultimate display]{deift2007universality}.
 \item This follows from \cite[pp 28 first display]{deift2007universality} using \eqref{eq: phi deriv}.
 \item This is due to \cite[(4.9)-(4.10) and points (1), (2), (4), (5) on page 29]{deift2007universality}.
 \item This is due to \cite[(4.14)]{deift2007universality} and the odd-/evenness of $\phi_N$.
 \item The first claim is due to \cite[display after (4.15)]{deift2007universality} and the second follows immediately using \eqref{eq: JN def}.
 \item This is due to \cite[(4.16)]{deift2007universality}.
\end{enumerate}
\end{proof}
When using \eqref{middle asymptotics} we will also use the asymptotics of the Airy function and its derivative.
\begin{lemma}{\cite[pp 448--449]{Abramowitz}}
  It holds that as $y\to \infty$,
  \aln{
  \begin{split}\label{asymptotics of Airy infty}
   &{\rm Ai}(y)\sim \frac{1}{2\sqrt{\pi}y^{1/4}} e^{-\frac{2}{3}y^{3/2}},\\
   &{\rm Ai}'(y)\sim -\frac{y^{1/4}}{2\sqrt{\pi}} e^{-\frac{2}{3}y^{3/2}},
   \end{split}\\
   \begin{split}\label{asymptotics of Airy -infty}
   &{\rm Ai}(-y)= \frac{1}{\sqrt{\pi}y^{1/4}} \sin{\left(\frac{2}{3}y^{3/2}+\frac{\pi}{4}\right)}+o(|y|^{-1/4}),\\
   &{\rm Ai}'(-y)= -\frac{y^{1/4}}{\sqrt{\pi}} \cos{\left(\frac{2}{3}y^{3/2}+\frac{\pi}{4}\right)}+o(|y|^{1/4}).
    \end{split}
  }
\end{lemma}

Due to the symmetry $\rho_N(x)=\rho_N(-x)$ it suffices to consider $x\geq 0$ in the proof of Lemma~\ref{lem:rhoapprox}. We furthermore consider $\e>0$ arbitrary, and chose a sufficiently small $\delta>0$ such that
 \aln{\label{choice of delta}
 \inf_{x\in[\sqrt{2}(1-\delta),\sqrt{2}(1+\delta)]}\Phi(x)>-\e/2.
 }
The claims of Lemma~\ref{lem:rhoapprox} then follow from the following two estimates.

    \begin{equation}\label{rhoN asympt 2}
      \rho_N(x)=
      \frac{\exp{(N\,\Phi(x))}}{2\sqrt{\pi\,N}\,(x^2-2)^{\frac 1 4}
        \,(|x|+\sqrt{x^2-2})^{\frac{1}{2}+o(1)}} \text{ uniformly for $x \in [\sqrt{2}(1+\delta),\infty)$},
    \end{equation}
    and
   \begin{equation}
      \label{Upper_bound_of_rho_N 2}
       e^{-N\frac \e 2} \leq \rho_N(x)\leq e^{N\frac{\e}{2}},\, \forall x\in\ [0, \sqrt{2} + \delta),
    \end{equation}
(\eqref{rhoN asympt 2} implies \eqref{Upper_bound_of_rho_N} in the range $[\sqrt{2}(1+\delta),\infty)$ since $$ |\ln( 2\sqrt{\pi\,N}\,(x^2-2)^{\frac 1 4}
        \,(|x|+\sqrt{x^2-2}))| \le  \e |\Phi(x)| N ,$$ uniformly for $x$ in the range, for $N$ large enough).
The proof of \eqref{Upper_bound_of_rho_N 2} is further subdivided into 4 subcases:
\begin{enumerate}
    \item $x \in [0,\sqrt{2} - \delta)$
    \item $x \in [\sqrt{2} - \delta,\sqrt{2} - N^{-4/7})$
    \item $x \in [\sqrt{2} - N^{-4/7},\sqrt{2} + N^{-4/7})$
    \item $x \in [\sqrt{2} + N^{-4/7},\sqrt{2} + \delta)$
\end{enumerate}

In the range $[\sqrt{2} + \delta),\infty)$ the term $B_N(x)$ is dominant, and is estimated using \eqref{outside asymptotics} to obtain \eqref{rhoN asympt 2}. In case 4 the term $B_N(x)$ remains dominant, but is now estimated instead using \eqref{middle asymptotics}.

In case 1 the term $A_N(x)$ is dominant and is estimated with \eqref{sc asymptotics}. In case 2 the term $A_N(x)$ remains dominant and is now estimated instead using \eqref{middle asymptotics}.

Finally in the the intermediate case 3 the terms $A_N(x)$ and $B_N(x)$ are of similar order and for the upper bound crudely bounding the Hermite polynomial terms using \eqref{middle asymptotics} suffices, while for the lower bound we use a different method involving the formula \eqref{explicit formula ev} to compare $\rho_N$ in this range with $\rho_N$ in the range $[\sqrt{2} - \delta,\sqrt{2} - N^{-4/7})$.

%\begin{lemma}
%  Uniformly for $x \ge \sqrt{2} - N^{-4/7}$ we have
%  \begin{equation}\label{eq: B est}
%      B_N(x) \sim \phi_{N-1}(x)
%  \end{equation}
%\end{lemma}
%\begin{proof}
%\end{proof}

\begin{proof}[Proof of \eqref{rhoN asympt 2}]
If $N$ is even we have $J_N(x)=\sqrt{N}\int^\infty_0 \phi_N(\sqrt{N} x)\dd x -\sqrt{N}\int^\infty_x \phi_N(\sqrt{N} x)\dd x\sim (2N)^{-1/4}$ uniformly by \eqref{eq: JN def}, \eqref{integral hermite fcn} and  \eqref{(4.16)}, and thus by \eqref{eq: B def}-\eqref{eq: alpha def}
\begin{equation}\label{eq: B est}
B_N(x) \sim \frac{(2N)^{1/4}}{2} \phi_{N-1}(\sqrt{N}x) \text{ uniformly for } x\ge \sqrt{2}+\delta.
\end{equation}
For odd $N$, $J_N(x)$ decays exponentially by \eqref{eq: JN def} and \eqref{(4.16)}. %\cite[Lemma 7.2]{auffinger2013random}.
Hence, $S_N(x) \ll \phi_{N-1}(\sqrt{N}x)$ by \eqref{eq: SN def}. On the other hand $\alpha_N(x) \sim  \frac{(2N)^{1/4}}{2} \phi_{N-1}(\sqrt{N}x)$ by \eqref{eq: alpha def} and \eqref{integral hermite fcn}, so \eqref{eq: B est} holds also for odd $N$.

We now derive an estimate for $\phi_{N-1} (\sqrt{N}x)$  from the estimate \eqref{outside asymptotics} for $\phi_{N} (\sqrt{N}x)$. To this end define the function $F(s)=s^{-1}\Phi(\sqrt{s} x)$.  Then, by the mean value theorem, there exists $\alpha\in[1,\frac{N}{N-1}]$ such that  $$\frac{1}{N-1}F'(\alpha)=F\left(\frac{N}{N-1}\right)-F(1)=\frac{N-1}{N}\Phi\left(\sqrt{\frac{N}{N-1}}\, x\right)-\Phi(x).$$
We note that $$F'(\alpha)=-\alpha^{-2}\ln{\left(\frac{\sqrt{\alpha}x+\sqrt{\alpha x^2-2}}{\sqrt{2}}\right)}=-(1+o(1))\ln{\left(\frac{x+\sqrt{ x^2-2}}{\sqrt{2}}\right)},$$
where $o(1)$ converges to $0$ as $\alpha\rightarrow 1$ or equivalently $N\to\infty$  uniformly for $x>\sqrt{2}(1+\delta)$. Since, $(|x+\sqrt{2}|^{1/2}+|x-\sqrt{2}|^{1/2})^2=2(x+\sqrt{x^2-2})$, we obtain $$g(x)=\frac{|x+\sqrt{2}|^{1/2}+|x-2|^{1/2}}{(x^2-2)^{1/4}}=\frac{\sqrt{2(x+\sqrt{x^2-2})}}{(x^2-2)^{1/4}}.$$
 Hence, by \eqref{outside asymptotics},
\aln{\label{eq:phiNminus1assymp}
\phi_{N-1}
(\sqrt{N}x)&=\phi_{N-1}
\left(\sqrt{N-1} \sqrt{\frac{N}{N-1}}x\right)\notag\\
&\sim \frac{e^{(N-1)\Phi(\sqrt{\frac{N}{N-1}}x)}g(x)
}{\sqrt{4\pi \sqrt{2(N-1)}}}\notag\\
&= \frac{e^{N\Phi(x)+
\frac{N}{N-1}F'(\alpha)}g(x)}{\sqrt{4\pi \sqrt{2(N-1)}}}= \frac{e^{N\Phi(x)}}{\sqrt{\pi \sqrt{2N}}(x^2-2)^{\frac 1 4}(x+\sqrt{x^2-2})^{\frac{1}{2}+o(1)}}.
}
%$$\alpha_N(x)\sim \frac{e^{N\Phi(x)}}{2\sqrt{\pi}(x^2-2)^{1/4}(|x|+\sqrt{x^2-2})^{\frac 1 2 +o(1)}}.$$
From this we immediately obtain
\al{
B_N(x)&= \frac{(1+o(1)) e^{N\Phi(x)}}{2\sqrt{\pi}(x^2-2)^{1/4}(|x|+\sqrt{x^2-2})^{\frac 1 2 +o(1)}}\\
&= \frac{e^{N\Phi(x)}}{2\sqrt{\pi}(x^2-2)^{1/4}(|x|+\sqrt{x^2-2})^{\frac 1 2 +o(1)}}.
}
Recalling \eqref{eq: rho A B} it thus only remains to show that $A_N(x) =  o(B_N(x))$. By applying
 \eqref{outside asymptotics} for $\phi_{N-1},\phi_{N}, \phi_{N+1}$ we get from \eqref{eq: A def}
$$A_N(x)=N\phi_N(\sqrt{N}x)^2-\sqrt{N(N+1)}\phi_{N-1}(\sqrt{N}x)\phi_{N+1}(\sqrt{N}x)=O(N^2 e^{2N\Phi(x)}),$$
which implies $A_N(x) = o(B_N(x))$, since $\Phi(x) \le -\varepsilon/2$ for $x \ge \sqrt{2} + \delta$.
\end{proof}

We next move to the region $[\sqrt{2}+N^{-4/7},\sqrt{2}+\delta)$. The argument is essentially the same as for $[\sqrt{2}(1+\delta),\infty)$ except for using \eqref{middle asymptotics} instead of \eqref{outside asymptotics}.

\begin{proof}[Proof of \eqref{Upper_bound_of_rho_N 2} for $x\in  [\sqrt{2}+N^{-4/7},\sqrt{2}+\delta)$.]
%We consider the region $\sqrt{2}(1+N^{-4/7}) \leq |x|\leq \sqrt{2}(1+\delta)$.
% We first suppose $N$ is even. By a similar argument of the proof in (4.12) [Deift,Univ] with $c_N=O(\sqrt{N})$,
%$$|S_N(x)|\leq \sqrt{\frac{N}{2}}|J_N(x)|\leq O(N^{-3/4}).$$
%  Moreover,
%  $$\sum_{i=1}^N \phi_i(\sqrt{N}x)^2= $$
%ere exists $C>0$ such that for $N$ large enough, $$C^{-1}\leq \inf_{x\in [\sqrt{2}(1-\delta),\sqrt{2}(1+\delta)}\hat{f}_N(x)\leq \sup_{x\in [\sqrt{2}(1-\delta),\sqrt{2}(1+\delta)}\hat{f}_N(x)\leq C.$$}
By \eqref{middle asymptotics} and \eqref{asymptotics of Airy infty}, we obtain
\aln{\label{aymp herm fcn middle}
\phi_N(\sqrt{N} x)=  \exp\left(-\frac{2^{7/4}}{3} N(x-\sqrt{2})^{3/2}\hat{f}_N(x)^{3/2}+O(\ln(N))\right)
}
uniformly on $\sqrt{2}(1+N^{-4/7}) \leq |x|\leq \sqrt{2}(1+\delta)$.
Hence, together with \eqref{(4.16)}, it is easy to see  that $\int_{|x|}^\infty \phi_N(\sqrt{N}t)\dd t$ decays faster than any polynomial of $N$.

Hence, for $N$ even, by \eqref{integral hermite fcn}, we have
$$J_N(x)=\sqrt{N}\left(\int_0^{\infty} \phi(\sqrt{N} x) \dd x-\int_{|x|}^\infty \phi(\sqrt{N} x) \dd x\right)\sim (2N)^{-1/4},$$
as in the proof for $x \in [\sqrt{2}+\delta,\infty)$ above,
so that still
\begin{equation}\label{eq: B_N est 2}
    B_N(x) \sim \frac{(2N)^{1/4}}{2} \phi_{N-1}(x).
\end{equation}
For $N$ odd we obtain $\alpha_N(x) \sim \frac{(2N)^{1/4}}{2} \phi_{N-1}(x)$ by \eqref{eq: alpha def} and \eqref{integral hermite fcn}. Also since \eqref{eq: JN def} gives that $|J_N(x)|=\left|\sqrt{N}\int_{|x|}^\infty \phi_N(\sqrt{N} t) \dd t\right|$ decays faster than any polynomial of $N$ we obtain from \eqref{eq: SN def} that $S_N(x) =o(\alpha_N(x))$, by \eqref{integral hermite fcn}, so that \eqref{eq: B_N est 2} holds also for odd $N$.

 By \eqref{eq: A def} and  \eqref{aymp herm fcn middle} we get \al{
 |A_N(x)|&=\exp\left(-\frac{2^{11/4}}{3} N(x-\sqrt{2})^{3/2}\hat{f}_N(x)^{3/2}+O(\ln(N))\right)\\ &\ll B_{N}(x)
 =\exp\left(-\frac{2^{7/4}}{3} B_N(x-\sqrt{2})^{3/2}\hat{f}_N(x)^{3/2}+O(\ln(N))\right) \to 0.}
 Putting things together with \eqref{choice of delta},  we have
$$e^{-N\e/2}\leq  \frac{1}{2} N^{-1/2}B_{N}(x)\leq \rho_N(x)\leq 2N^{-1/2}B_{N}(x)\leq 1\leq e^{N\e/2},$$
concluding the proof.
\end{proof}

\begin{proof}[Proof of \eqref{Upper_bound_of_rho_N 2} for $x\in  [0,\sqrt{2}-\delta)$.]
It holds that $N^{-1/2} A_N(x) = e^{O(1)}$ uniformly in this interval by \eqref{sc asymptotics}. Using \eqref{inside asymptotics} as well as  \eqref{integral hermite fcn} and \eqref{estimate JN}, we obtain $S_N=O(1)$ and $\alpha_N(x)=O(1)$, so that $N^{-1/2}B_N(x) = o(1)$, which gives the claim.
\end{proof}

\begin{proof}[Proof of \eqref{Upper_bound_of_rho_N 2} for $x\in  [\sqrt{2}-\delta,\sqrt{2}-N^{-4/7})$.] By \eqref{middle asymptotics} and \eqref{asymptotics of Airy -infty}, we have % and the boundedness of the Airy function,
\aln{\label{upper hermite}
  \phi_{n}(\sqrt{N} x)=O(N^{-1/4} (x-\sqrt{2})^{-1/4} ) = o(1),\,\text{ for }n\in\{N-1,N,N+1\},
}
uniformly on $x\in [\sqrt{2}(1-\delta),   \sqrt{2}(1-N^{-4/7})]$.
Hence $|A_N(x)|,\,|S_N(x)|,\,|\alpha_N(x)|=O(N^2)$ by \eqref{eq: A def}, \eqref{eq: SN def}, \eqref{estimate JN}, \eqref{eq: alpha def} and \eqref{integral hermite fcn}. The upper bound follows directly.

For the lower bound we use that
\al{
  A_N(x) \geq \sum_{k=\lf \frac{x^2 N}{2}\rf}^{\lf \frac{x^2 N}{2}+N^{\frac{1}{3}-\frac{1}{84}}\rf} \phi_k(\sqrt{N}x)^2
}
since $ \frac{x^2 N}{2}+N^{\frac{1}{3}-\frac{1}{84}}\leq N$.
For
$k\in \{\lf \frac{x^2 N}{2}\rf,\cdots, \lf \frac{x^2 N}{2}+N^{\frac{1}{3}-\frac{1}{84}}\rf\}$ we have
$$\sqrt{\frac{N}{k}}\, x = \sqrt{ \frac{Nx^2}{Nx^2/2 + O(N^{1/3-1/84})}} =\sqrt{2}+O(N^{-\frac{2}{3}-\frac{1}{84}}),$$ which implies for $f_k$ from \eqref{middle asymptotics} that $f_k\left(\sqrt{\frac{N}{k}}\, x\right)=o(1)$, and therefore by that estimate and since ${\rm Ai}(0)>0$ as well as
${\rm Ai}'(0)<0$
we have
\al{
\phi_k(\sqrt{N} x) = \phi_k\left(\sqrt{k}\,\left(\sqrt{\frac N k}\, x\right)\right)
\geq  c' N^{-\frac{1}{12}} {\rm Ai}\left(f_k\left(\sqrt{\frac{N}{k}}x\right)\right)\geq c'' N^{-\frac{1}{12}}, }
with some constants $c',c''>0$. Since $ \frac{x^2 N}{2}+N^{\frac{1}{3}-\frac{1}{84}}\leq N$ this implies
\al{
  A_N(x)\ge N^{\frac{1}{3}-\frac 1 {84}} N^{-\frac{1}{6}}=c N^{\frac{13}{84}},
}
for some $c>0$.
Furthermore $S_N(x)=O(N^{1/7})$ in this interval by \eqref{eq: SN def}, \eqref{estimate JN} and \eqref{upper hermite}. Similarly $\alpha_N(x)=O(N^{1/7})$ by \eqref{eq: alpha def}, \eqref{integral hermite fcn} and \eqref{upper hermite}. Since $N^{\frac{13}{84}}\gg N^{\frac 1 7}$, we obtain that $|S_N(x)|,|\alpha_N(x)|\ll A_N(x)$ uniformly on $x\in \mathcal{I}_N$ and hence,
\begin{equation}\label{eq: rho lb}
    \rho_N(x) \geq N^{-1} \text{ for } N \text{ large enough and } x\in [\sqrt{2}(1-\delta),\sqrt{2}(1-N^{-4/7})].
\end{equation}
\end{proof}

\begin{proof}[Proof of \eqref{Upper_bound_of_rho_N 2} for $x\in  [\sqrt{2}-N^{-4/7},\sqrt{2}-N^{-4/7})$.]
We write $a_N=N^{-4/7}$ and consider the region $x\in [\sqrt{2}(1-a_N),\sqrt{2}(1+a_N)]$. Since for $f_N$ as in \eqref{middle asymptotics} it holds $f_N(x)= O(N^c)$, we have  $|{\rm Ai}(f_N(x))|,\,|{\rm Ai}'(f_N(x))|=O(N^c)$ ,  uniformly for $x\in[\sqrt{2}(1-a_N),\sqrt{2}(1+a_N)]$ by \eqref{asymptotics of Airy infty} and \eqref{asymptotics of Airy -infty}. Hence $\phi(\sqrt{N} x)= O(N^c)$ uniformly using \eqref{middle asymptotics}. This implies that $|A_N(x)|,\,|S_N(x)|,\,|\alpha_N(x)|= O(N^c)$ using \eqref{eq: A def}, \eqref{eq: SN def}, \eqref{estimate JN}, \eqref{eq: alpha def}, \eqref{integral hermite fcn}. Therefore,
$\rho_N(x) \leq e^{N\e/2}$,
uniformly on $x\in [\sqrt{2}-a_N,\sqrt{2}+a_N]$ and for $N$ large enough by \eqref{eq: rho A B}. This gives the upper bound bound.

To obtain the lower bound we compare $\rho_N(x)$ to $\rho_N(y)$ for $y \le \sqrt{2} - N^{-4/7}$ (a case already covered above) as follows. From \eqref{explicit formula ev} and \eqref{rhoN formula} we have
\begin{equation}
   \rho_N(x) =\frac{1}{N}\sum_{j=1}^N \frac{N!}{Z_{N}} \int_{\mathbf{T}_{N}}  \delta(x_j-x) e^{-\frac{N}{2}\sum_{i=1}^N
      x_i^2}\,\Delta_N({\bf{x}})\,\prod_{i=1}^N \dd x_i.
\end{equation}
For any $z$ we can employ the change of variable $x_i\to x_i+x-z$ to obtain
\begin{equation}
   \rho_N(x)= \frac{(N-1)!}{Z_{N}}\sum_{j=1}^N \int_{\mathbf{T}_{N}}  \delta(x_j-z) e^{-\frac{N}{2}\sum_{i=1}^N
      (x_i+x-z)^2}\,\Delta_N({\bf{x}})\,\prod_{i=1}^N \dd x_i,
\end{equation}
noting that $\Delta_N((x_i))=\Delta_N((x_i+x-z))$. Integrating both sides over $z$ this gives
\begin{align*}
   \rho_N(x)&= (\sqrt{2}a_N)^{-1} \frac{(N-1)!}{Z_{N}}\sum_{j=1}^N \int_{\sqrt{2}(1-2a_N)}^{\sqrt{2}(1-a_N)}\dd z \int_{\mathbf{T}_{N}}  \delta(x_j-z) e^{-\frac{N}{2}\sum_{i=1}^N
      (x_i+x-z)^2}\,\Delta_N({\bf{x}})\,\prod_{i=1}^N \dd x_i\\
      &\geq \frac{(N-1)!}{ Z_{N}}\sum_{j=1}^N  \int_{\sqrt{2}(1-2a_N)}^{\sqrt{2}(1-a_N)}\dd z \\
      &\qquad\qquad\times \int_{\mathbf{T}_{N}} \delta(x_j-z) e^{-\frac{N}{2}\sum_{i=1}^N
     x_i^2- N(x-z)\sum x_i-\frac{N^2}{2}(x-z)^2 }\,\Delta_N({\bf{x}})\prod_{i=1}^N \dd x_i,
\end{align*}
Let $D_N = \left\{(x_i)_{i=1}^N:~\left|\sum_{i=1}^N x_i\right|\leq N a_N\right\}$.  Since $x\in [\sqrt{2}-a_N,\sqrt{2}+a_N]$ we have $|z-x|\leq 4 a_N$. Hence we can further bound from below by
\begin{align*}
      &\quad  e^{-c N^2 a_N^2} \frac{(N-1)!}{Z_{N}}\sum_{j=1}^N \int_{\sqrt{2}(1-2a_N)}^{\sqrt{2}(1-a_N)}\dd z  \int_{\mathbf{D}_N\cap \mathbf{T}_{N}}  \delta(x_j-z) e^{-\frac{N}{2}\sum_{i=1}^N
      x_i^2 }\,\Delta_N({\bf{x}})\, \prod_{i=1}^N \dd x_i\\
      &=  e^{-c N^2 a_N^2}\mathbb{E}\left[\frac 1 N \#\left\{i\in\{1,\cdots,N\}:\lambda^N_i\in[\sqrt{2}(1-2a_N),\sqrt{2}(1-a_N)]\right\}\mathbf{1}_{D_N}(\lambda^N)\right]\\
      &\geq e^{-c N^2 a_N^2}\left(  \int_{\sqrt{2}(1-2a_N)}^{\sqrt{2}(1-a_N)} \rho_N(z)\dd z-   \P\left(\left|\sum_{i=1}^N\lambda_i^N\right|> N a_N\right)\right),
\end{align*}
where  $\lambda^N=(\lambda_i^N)_{i=1}^N$ is the eigenvalues of ${\rm GOE}_N(N^{-1})$ as before. Using \eqref{eq: rho lb} we then have
$$\int_{\sqrt{2}(1-2a_N)}^{\sqrt{2}(1-a_N)} \rho_N(z)\dd z\geq  N^{-1} (\sqrt{2}a_N)\geq N^{-2}.$$
On the other hand, for ${\rm GOE}_N(N^{-1})=(A_{ij})_{1\leq ij\leq N}$,
\al{
   \P\left(\left|\sum_{i=1}^N\lambda_i^N\right|> N a_N\right)&= \P\left(\left|\sum_{i=1}^N A_{ii}\right|> N a_N\right)= O\left(e^{-\frac{N^2 a_N^2}{2}}\right),
}
since $(A_{ii})$ are centered Gaussian random variables of variance $N^{-1}$.
Therefore, we obtain
$$\rho_N(x)\geq \frac{1}{2 N^{2}} e^{-c N^2 a_N^2}=  \frac{1}{2 N^{2}} e^{-c N^{\frac{6}{7}}}\geq e^{-\e N}.$$
\end{proof}

\section{Covariances of the Hamiltonian}
The next lemma gives the covariances of the Hamiltonian $H_N$ (without
  the external field). For its proof see \cite[Lemma~1]{AB13} or
\cite[Appendix~A]{arous2020geometry}.

\begin{lemma}
  \label{lem:covariances_of_HN}
  For $1\leq i\leq j\leq N-1$, $1\leq \ell\leq k\leq N-1$ and $\sigma\in
  S_{N-1}$, we have:
  \begin{align*}
    \E[H_N(\sigma)H_N(\sigma)]&=N,\\
    \E[\partial_i H_N(\sigma)H_N(\sigma)]&=0,\\
    \E[\partial_{ij} H_N(\sigma)H_N(\sigma)]&=0,\\
    \E[\partial_i H_N(\sigma) \partial_\ell H_N(\sigma)]&=N\,\xipone\,\delta_{i\ell},\\
    \E[\partial_{ij} H_N(\sigma) \partial_\ell H_N(\sigma)]&=0,\\
    \E[\partial_{ij} H_N(\sigma)\partial_{\ell k}
      H_N(\sigma)]&=N\,\xione\,\delta_{i\ell}\delta_{j k}(1+\delta_{ij}),\\
    \E[\dN H_N(\sigma)\dN H_N(\sigma)]&=N\,(\xipone+\xippone),\\
    \E[H_N(\sigma)\dN H_N(\sigma)]&=N\,\xipone,\\
    \E[\partial_i H_N(\sigma)\dN H_N(\sigma)]&=0,\\
    \E[\partial_{ij} H_N(\sigma) \dN H_N(\sigma)]&=0.
  \end{align*}
\end{lemma}
\begin{proof}
  Use \cite[(5.5.4)]{adler2009random} with $H=H_N$ and
  $C(\sigma,\tau)=\xi(\sigma\cdot \tau)$.
\end{proof}
\printbibliography
\end{document}